\documentclass[12pt, reqno]{amsart}
\usepackage{amscd,amsmath,amsthm,amssymb,graphics}
\usepackage{amsfonts,amssymb,amscd,amsmath,enumitem,verbatim}
\usepackage[a4paper,top=3cm,left=3cm,right=3cm]{geometry}
\usepackage{subcaption}
\usepackage{xcolor}
\usepackage{float}
\usepackage{graphicx}
\theoremstyle{plain}
\usepackage[linesnumbered, ruled]{algorithm2e}
\usepackage{booktabs}
\usepackage{comment}
\newtheorem{Theorem}{Theorem}
\newtheorem{Lemma}[Theorem]{Lemma}
\newtheorem{Corollary}[Theorem]{Corollary}
\newtheorem{Proposition}[Theorem]{Proposition}

\newtheorem{Definition}[Theorem]{Definition}
\newtheorem{Remark}[Theorem]{Remark}

\newtheorem{Example}[Theorem]{Example}

\newcommand{\Z}{\mathbb{Z}}
\newcommand{\N}{\mathbb{N}}
\newcommand{\Q}{\mathbb{Q}}
\newcommand{\R}{\mathbb{R}}

\newcommand{\F}{\mathbb{F}}

\usepackage[backend=biber,maxbibnames=99]{biblatex}

\title{The arithmetic of continued fractions in the field of $p$-adic numbers}


\author[G. Romeo]{Giuliano Romeo}
\address{Department of Mathematical Sciences ``Giuseppe Luigi Lagrange", Politecnico di Torino; Charles University, Faculty of Mathematics and Physics, Department of Algebra, Sokolovská 83, 186 00 Praha 8, Czech Republic}
\email{giuliano.romeo@polito.it\\
romeo.giuliano@matfyz.cuni.cz}
\author[G. Salvatori]{Giulia Salvatori}
\address{Department of Mathematical Sciences ``Giuseppe Luigi Lagrange", Politecnico di Torino}
\email{giulia.salvatori@polito.it}

\usepackage{hyperref}
\begin{document}

\begin{abstract}

Continued fractions have been long studied due to their strong properties, such as rational approximation. In this extent, their arithmetic over real numbers has represented an intriguing problem throughout the years. In this paper, we develop the arithmetic of continued fractions over the field of $p$-adic numbers. In particular, we provide a complete methodology to compute the $p$-adic continued fraction of the M\"obius transformation and the bilinear fractional transformation of $p$-adic numbers. These allow any standard arithmetic operation over $p$-adic numbers to be performed. In great contrast with real continued fractions, we prove that the knowledge of arbitrarily many partial quotients of the initial continued fractions is not always sufficient to recover some partial quotients of the transformations. However, we prove that the set of elements for which this is not possible has Haar measure zero in $\Q_p$.
\end{abstract}

\maketitle
\tableofcontents

\section{Introduction}
Continued fractions over real numbers have been studied for centuries and employed in different areas of mathematics, mainly due to their excellent approximation properties. They are expressions of the form
\begin{equation}\label{Eq: continued fraction}
a_0+\cfrac{1}{a_1+\cfrac{1}{a_2+\ddots}} = [a_0,a_1, a_2, \ldots],
\end{equation}
where the coefficients $a_n$ are called \textit{partial quotients}. Any real number $\alpha\in\R$ can be represented as a continued fraction where the partial quotients are computed, for all $n\geq0$, starting from $\alpha_0=\alpha$, as:

\begin{equation}\label{Alg: R}
\begin{cases}
a_n=\lfloor \alpha_n \rfloor \\
\alpha_{n+1}=\frac{1}{\alpha_n-a_n},
\end{cases}
\end{equation}
where $\lfloor \cdot \rfloor $ denotes the \textit{floor function} and the elements $\alpha_n=[a_n,a_{n+1},\ldots]$ are called \textit{complete quotients}. If $\alpha_n=a_n$ for some $n\in\mathbb{N}$, then the algorithm terminates and the continued fraction is finite. For more details about the properties and the general theory of continued fractions we refer the reader to \cite{khinchin1964continued,olds1963continued,wall2018analytic}.
Continued fractions provide another way to represent real numbers that is, for some aspects, more efficient than the usual decimal representation. In fact, in great contrast with the decimal or the base-$b$ expansion, a real number is rational if and only if its continued fraction is finite and it is a quadratic irrational if and only if its continued fraction is eventually periodic (the famous Lagrange's theorem). However, the reason why we commonly use base-$b$ expansions is that they lead to a simple arithmetic that allows to perform efficiently operations among real numbers. Developing an arithmetic using the partial quotients of continued fraction representations of real numbers is more complicated and it represents an intriguing problem. For example, it is not possible to determine the first partial quotient of $2\alpha$, that is $\lfloor2\alpha\rfloor$, knowing only the first partial quotient of $\alpha$, that is $\lfloor\alpha\rfloor$. The continued fraction of $2\alpha$ has been studied by Hurwitz \cite{hurwitz1963kettenbruch}. Then, Hall \cite{hall1947sum} studied in more generality operations among continued fractions. Lately, further improvements have been carried out by Cusick \cite{cusick1971sums} and Raney \cite{raney1973continued}. In 1972, Gosper \cite{gosper} came out with a very efficient algorithm in order to compute the continued fraction expansion of any linear fractional transformation (M\"obius transformation) of a real number $\alpha$, that is
\begin{equation*}
\gamma =\frac{x\alpha+y}{z\alpha+t},
\end{equation*}
with $x,y,z,t\in\Q$  and the bilinear fractional transformation of two real numbers, that is
\begin{equation*}
\gamma=\frac{x\alpha\beta+y\alpha+z\beta+t}{e\alpha\beta+f\alpha+g\beta+h},
\end{equation*}
with $x,y,z,t,e,f,g,h\in\Q$. These transformations include, as special cases, all arithmetic operations on one or two continued fractions.
Gosper's algorithms for computing the continued fractions of M\"obius and bilinear fractional transformations are described in Sections~\ref{Sec: Mobius} and~\ref{Sec: bilinear}, respectively. This analysis has been lately extended by Liardet and Stambul \cite{liardet1998algebraic}. Moreover, Lagarias and Shallit \cite{lagarias1997linear}  proved some bounds on the continued fraction of the M\"obius transformation, provided that the starting real number has bounded partial quotients (see also the work of Stambul \cite{stambul2000continued}). Some results about the continued fraction expansion of linear fractional transformations can be found in \cite{havens2020linear}. Moreover, these transformations naturally appear in the study of the modular group, whose action on the upper-half plane and its geodesics are connected to continued fractions \cite{series1985modular}. Gosper's algorithm has been recently generalized for the computation of M\"obius and bilinear fractional transformations of multidimensional continued fractions \cite{miska2025arithmetic}.
The aim of this paper is to address similar questions for continued fractions defined over the field of $p$-adic numbers $\Q_p$. Continued fractions over $\Q_p$ have been introduced by Mahler \cite{mahler1940geometrical} in 1940. Unlike real continued fractions, there is no standard algorithm for computing $p$-adic continued fractions. In fact, there is not a unique satisfactory way to choose a $p$-adic floor function to replicate Algorithm~\eqref{Alg: R} over $p$-adic numbers. There are a few natural definitions of $p$-adic continued fractions. The earliest and most commonly used are the algorithms of Ruban \cite{rub} and Browkin \cite{browkin1978continued}, introduced around 1970, and forming the main topic of this work. Another algorithm has been defined by Schneider \cite{schneider1968p} for continued fractions that are not simple, i.e. the numerators in \eqref{Eq: continued fraction} are not necessarily $1$. The problem of finding an algorithm to produce $p$-adic continued fractions with the same good properties of continued fractions in $\R$ is still open. In fact, Ruban's and Schneider's continued fractions are not always finite for rational numbers \cite{laohakosol1985characterization,bundschuh1977p} and they are not always eventually periodic for $p$-adic quadratic irrationals \cite{capuano2019effective,tilborghs1990periodic}.
Browkin’s algorithms \cite{browkin1978continued,browkin2001continued} yield finite continued fractions for every rational number \cite{browkin1978continued,barbero2021periodic}, as in the real case. The problem of deciding whether any $p$-adic quadratic irrational has a periodic Browkin's continued fraction is still open. In fact, despite several progress has been made \cite{bedocchi1988nota,bedocchi1989remarks,capuano2023periodicity,  capuano2019effective,murru2023periodicity,romeo2024real}, Lagrange's theorem has not been proved nor disproved for Browkin's continued fractions. Recently, other $p$-adic continued fractions algorithms have been defined in order to gain better properties of periodicity \cite{barbero2024periodic,murru2024new,murru2023convergence,wang2024convergence,yasutomi2025simultaneous}. For a survey on the general theory of $p$-adic continued fractions see \cite{romeo2024continued}.\bigskip

In this paper we develop the study of the arithmetic of continued fractions in the field of $p$-adic numbers. We provide effective methods to compute the M\"obius transformation and the bilinear fractional transformation of $p$-adic continued fractions. Both the results and the techniques are different from those in the real setting, requiring new arguments. In Section~\ref{Sec: preliminaries}, we recall some known facts about continued fractions in $\mathbb{R}$ and $\mathbb{Q}_p$, as well as the classical algorithms for computing the continued fraction of M\"obius and bilinear fractional transformations. In Section~\ref{Sec: measure}, we recall some known metric results and prove additional ones for the main algorithms for $p$-adic continued fractions. 
The main result is that, for the three $p$-adic continued fraction expansions considered in this work, the partial quotients of $\mu$-\textit{almost all} $p$-adic numbers (where $\mu$ is the Haar measure) have unbounded $p$-adic valuations, with their exact distribution fully determined.
Then, in Section~\ref{Sec: aux}, we prove several lemmas on the arithmetic of $p$-adic numbers, in order to understand how the information of the known $p$-adic digits propagates after simple arithmetic operations. The results of Sections~\ref{Sec: measure} and~\ref{Sec: aux} are useful in the proofs of the main results, but they can also be read independently from the rest of the paper. In Section~\ref{Sec: Mobpadic}, we collect all the results concerning the M\"obius transformation of $p$-adic continued fractions for the case of Ruban and \textit{Browkin I} algorithms. Theorem~\ref{thmfloor} provides necessary and sufficient conditions to determine the $p$-adic floor function of the M\"obius transformation of $\alpha \in \Q_p$ from the coefficients of the transformation and the first partial quotient of $\alpha$. If the hypotheses of Theorem~\ref{thmfloor} are not satisfied, the most natural strategy is to use more partial quotients of $\alpha$ by performing \textit{input transformations} (see Section~\ref{Sec: MobInp} for more details). We show that the optimal situation in this case occurs when we have 
\begin{equation}\label{Eq: verygoodcondition}
v_p(x\alpha)<v_p(y)\quad \text{and} \quad  v_p(z\alpha)<v_p(t).
\end{equation}
In fact, whenever \eqref{Eq: verygoodcondition} is satisfied, we are able to compute the partial quotient of the M\"obius transformation if and only if $v_p(a_n)\leq k$, where $k=v_p(x)-v_p(z)$ is a constant quantity and $\{a_n\}_{n\in\N}$ is the sequence of partial quotients of $\alpha$. One of the most interesting outcome of our analysis is that the fulfillment of the output condition for the computation of the floor function only depends on a single partial quotient. This differs significantly from the real case, where all the partial quotients contribute to give a better approximation of the M\"obius transformation and help to satisfy the condition. In Lemma~\ref{Lem: rec} we prove that after an input transformation, then \eqref{Eq: verygoodcondition} is satisfied, unless either
\begin{equation}\label{Eq: verybadcondition}
v_p(x\alpha+y)\geq v_p(x)+1 \quad \text{or} \quad  v_p(z\alpha+t)\geq v_p(z)+1.
\end{equation}
It turns out, by Corollary~\ref{Cor: valunique} and Proposition~\ref{Prop: tolet}, that \eqref{Eq: verybadcondition} cannot be satisfied at all steps, unless $\alpha$ is the root of either the numerator or the denominator of the transformation. It means that, apart from these cases, after using finitely many partial quotients of $\alpha$, condition~\eqref{Eq: verygoodcondition} is guaranteed to be satisfied. In addition, by performing output transformations we do not end up in this undesired situation, as shown in Proposition~\ref{Pro: outgood}. The striking difference with classical continued fractions is that the output condition is not guaranteed to be always satisfied after performing an arbitrary number of input transformations. A construction for such case is provided in Example~\ref{Exa: neveroutput}. This means that the $p$-adic floor function of a transformation of $\alpha\in\Q_p$ cannot be always recovered by the knowledge of an arbitrary number of partial quotients of $\alpha$. Therefore, in great contrast with real continued fractions, it is not possible to develop a complete arithmetic on $p$-adic numbers by means of their continued fraction expansions. However, by the metric results of Section~\ref{Sec: measure}, we know that the set of such $\alpha\in\Q_p$ has Haar measure zero. In Section~\ref{Sec: bilinearfractional}, we study the $p$-adic continued fraction of the bilinear fractional transformation of two $p$-adic numbers $\alpha,\beta\in\Q_p$, for the case of Ruban and \textit{Browkin I} algorithms. As in Section \ref{Sec: Mobpadic}, we prove necessary and sufficient conditions to determine its $p$-adic floor function. In particular, this characterization is provided in Theorem~\ref{thmfloorBilinear}. If the following inequalities hold:
\begin{align*}
&v_p(x\beta)<v_p(y),  & 
&v_p(z\beta)<v_p(t),   &
&v_p(x\alpha)<v_p(z ),\\
&v_p(e \beta)<v_p(f ),   &
&v_p(g \beta)<v_p(h ),   &
&v_p(e \alpha)<v_p(g ),
\end{align*}
then we show in Remark~\ref{rem: outputcond} that the output condition of Theorem~\ref{thmfloorBilinear} heavily simplifies, becoming
\begin{equation}\label{Eq: verycondab}
\min\{-v_p(a_n),-v_p(b_n)\}\geq u,
\end{equation}
for some $n\in\N$, where $u=v_p(e)-v_p(x)$ is a constant quantity and $\{a_n\}_{n\in\N}$ and $\{b_n\}_{n\in\N}$ are the sequences of partial quotients of $\alpha$ and $\beta$, respectively. Therefore, the fulfillment of the hypotheses of Theorem~\ref{thmfloorBilinear} again depends on whether the partial quotients of $\alpha$ and $\beta$ have sufficiently negative valuation, and this is true for \textit{almost all} $p$-adic numbers. In Section~\ref{Sec: themobalgo} and Section~\ref{Sec: thebilalgorithm}, we discuss the implementations of our procedures for the computation of the $p$-adic continued fraction of the M\"obius transformation and the bilinear fractional transformation. These are, respectively, Algorithm~\ref{Alg: Gosp1} and Algorithm~\ref{Alg: Gosp2} in the Appendix.  The SageMath implementation of the two algorithms is publicly available\footnote{\href{https://github.com/giulianoromeont/p-adic-continued-fractions}{https://github.com/giulianoromeont/p-adic-continued-fractions}}.
In Section~\ref{Sec: otheralgo}, we take in consideration the analysis of the previous two sections for the case of the algorithm for $p$-adic continued fraction expansion developed in \cite{murru2024new}. Finally, in Section~\ref{Sec: Computations}, we present computational experiments and an analysis of the two algorithms’ performance, focusing on the number of input partial quotients needed to satisfy the output conditions.

\section{Preliminaries}\label{Sec: preliminaries}
In this section, we recall some basic facts and notation regarding continued fractions, the Möbius transformation, and the bilinear fractional transformation. For further background, we refer the reader to \cite{khinchin1964continued, olds1963continued, wall2018analytic} for continued fractions and to \cite{gouvea2020p, koblitz2012p} for $p$-adic numbers.

Let us denote by $p$ a prime number and by $v_p(\cdot)$ the $p$-adic valuation. Let $|\cdot|$ and $|\cdot|_p$ be, respectively, the standard Euclidean absolute value and the $p$-adic absolute value. Let us denote a simple continued fraction as in \eqref{Eq: continued fraction}
by $[a_0,a_1,a_2,\ldots]$. For all $n\in\mathbb{N}$, the rational numbers
\[\frac{A_n}{B_n}=[a_0,a_1,\ldots,a_{n-1},a_n]=
a_0 + \cfrac{1}{\displaystyle a_1 + \cfrac{1}{\displaystyle \ddots  a_{n-1}+\cfrac{1}{a_n}}},\]
are called the \textit{convergents} of the continued fraction. The sequences $\{A_n\}_{n\in\mathbb{N}}$ and $\{B_n\}_{n\in\mathbb{N}}$ of numerators and denominators of the convergents satisfy the following recursions:
\begin{equation*}
\begin{cases}
A_0=a_0,\\
A_1=a_1a_0+1,\\
A_n=a_nA_{n-1}+A_{n-2}, \ \ n \geq 2,
\end{cases} \quad 
\begin{cases}
B_0=1,\\
B_1=a_1,\\
B_n=a_nB_{n-1}+B_{n-2}, \ \ n \geq 2.
\end{cases}    
\end{equation*}

\subsection{M\"obius transformation}\label{Sec: Mobius}
Let us recall the idea of Gosper's algorithm for the computation of the partial quotients of the M\"obius transformation of a continued fraction in $\R$. For further details on the algorithm, we refer the reader to \cite[Appendix 2]{gosper}. Let us consider $x,y,z,t \in \Q$ and let $\alpha=[a_0,a_1,\ldots]$, where the partial quotients are integers such that $a_i\geq 1$ for all $i\geq1$. The M\"obius transformation, or linear fractional transformation, of $\alpha$ is the function
\begin{equation}\label{Eq: Mobius}
\gamma =\frac{x\alpha+y}{z\alpha+t},
\end{equation}
such that $xt-yz\neq0$. Gosper's algorithm takes as input $x,y,z,t$ and the sequence of partial quotients $\{a_n\}_{n\geq0}$, and computes the sequence of partial quotients $\{l_n\}_{n\geq0}$ of the continued fraction of \eqref{Eq: Mobius}. The idea in the field of real numbers is that we are able to determine the floor function $\left\lfloor\frac{x\alpha+y}{z\alpha+t}\right\rfloor$ if and only if
\begin{equation}\label{Eq: condition_Gosper}
\left\lfloor\frac{x}{z} \right\rfloor=\left\lfloor\frac{x+y}{z+t} \right\rfloor \quad \text{and} \quad  sign(z)=sign(z+t).
\end{equation}
We can assume $\alpha\geq1$, because, if $\alpha<1$, we replace $\alpha$ with $a_0 + \frac{1}{\alpha_1}$. 
In this case, if $sign(z)=sign(z+t)$, the function $z \alpha + t$ has no zeros on the interval $[1, + \infty)$. Hence, $f(a) = \frac{xa + y}{za + t}$ is well-defined, continuous and monotone on $[1,+\infty)$.
This implies that either
\[\frac{x}{z}\leq\frac{x\alpha+y}{z\alpha+t} \leq \frac{x+y}{z+t},\]
or
\[\frac{x+y}{z+t}\leq\frac{x\alpha+y}{z\alpha+t} \leq \frac{x}{z}. \]
If condition~\eqref{Eq: condition_Gosper} is satisfied, then necessarily
\[l_0=\left\lfloor\frac{x\alpha+y}{z\alpha+t} \right\rfloor=\left\lfloor\frac{x}{z} \right\rfloor=\left\lfloor\frac{x+y}{z+t} \right\rfloor \quad \forall \ \alpha \ge 1.\]
Therefore, we can perform the \textit{output transformation} and compute the next complete quotient of the M\"obius transformation, that is
\begin{equation*}
\gamma_1=\frac{1}{\gamma-l_0}=\frac{1}{\frac{x\alpha+y}{z\alpha+t}-l_0}=\frac{z\alpha+t}{(x-l_0z)\alpha+(y-l_0t)}.
\end{equation*}
If condition~\eqref{Eq: condition_Gosper} is not satisfied, we are not able to compute the partial quotient of the M\"obius transformation. In this case, we perform the \textit{input transformation}, by using the first partial quotient of the continued fraction of $\alpha$. We can write 
\begin{equation}\label{Eq: inputtransformationsec2}
\cfrac{x\alpha+y}{z\alpha+t} = \cfrac{x\left(a_0+\frac{1}{\alpha_1}\right)+y}{z\left(a_0+\frac{1}{\alpha_1}\right)+t}=\cfrac{(xa_0+y)\alpha_1 + x}{(za_0+t)\alpha_1 + t}. 
\end{equation}
Both the input and output transformations can be more concisely expressed in matrix form. In fact, the matrix
\[M=\begin{pmatrix}
    x & y \\
    z & t
\end{pmatrix}\]
is transformed into
\[\begin{pmatrix}
    x & y \\
    z & t
\end{pmatrix}\begin{pmatrix}
    a_0 & 1 \\
    1 & 0
\end{pmatrix}=\begin{pmatrix}
    xa_0+y & x \\
    za_0+t & z
\end{pmatrix},\]
by the input transformation, and into
\begin{equation}\label{Eq: outputreal}
\begin{pmatrix}
    0 & 1 \\
    1 & -l_0
\end{pmatrix}\begin{pmatrix}
    x & y \\
    z & t
\end{pmatrix}=\begin{pmatrix}
    z & t \\
    x-l_0z& y-l_0t
\end{pmatrix},
\end{equation}
by the output transformation. It can be showed that, after a finite number of input transformations, condition~\eqref{Eq: condition_Gosper} is satisfied, and the partial quotient $l_0=\lfloor\gamma\rfloor$ can be computed correctly.

\subsection{Bilinear fractional transformation}\label{Sec: bilinear}In this section, we deal with the bilinear fractional transformation of two continued fractions (for more details, see \cite[Appendix 2]{gosper}). Consider $x, y, z, t, e, f, g, h \in \mathbb{Q}$, and let $\alpha, \beta \in \mathbb{R}$ be real numbers with continued fraction expansions $\alpha = [a_0, a_1, \ldots]$ and $\beta = [b_0, b_1, \ldots]$, where the partial quotients $a_i$ and $b_i$ are integers satisfying $a_i, b_i \geq 1$ for all $i \geq 1$. The bilinear fractional transformation of $\alpha$ and $\beta$ is the function
\begin{equation}\label{Eq: bilinear}
\gamma=\frac{x\alpha\beta+y\alpha+z\beta+t}{e\alpha\beta+f\alpha+g\beta+h},
\end{equation}
such that the matrix
$\begin{pmatrix}
    x & y & z & t \\
    e & f & g & h
\end{pmatrix}$
has full rank $2$. As in the case of M\"obius transformation, if some conditions are satisfied, it is possible to determine the floor function $\lfloor \gamma\rfloor$. In particular, $\lfloor \gamma\rfloor$ is uniquely determined whenever
\begin{equation}\label{Eq: condition_bilinear}
\left\lfloor \frac{x}{e}\right\rfloor = \left\lfloor \frac{x+y}{e+f}\right\rfloor = \left\lfloor \frac{x+z}{e+g}\right\rfloor = \left\lfloor \frac{x+y+z+t}{e+f+g+h}\right\rfloor = l_0,
\end{equation}
and $e$, $e+f$, $e+g$, $e+f+g+h$ have all the same sign. In this case, the floor function of the transformation is uniquely determined: $\lfloor \gamma \rfloor = l_0$, which is the first partial quotient of the continued fraction expansion of $\gamma$. As for the M\"obius transformation, we can perform the output transformation to compute the following complete quotient, that is
\begin{equation}
\gamma_1=\frac{1}{\gamma-l_0}=\frac{e\alpha\beta+f\alpha+g\beta+h}{(x-l_0e)\alpha\beta+(y-l_0f)\alpha+(z-l_0g)\beta+(t-l_0h)},
\end{equation}
and it can be written in matrix form as 

\[\begin{pmatrix}
    0 & 1 & \\
    1 & -l_0 
\end{pmatrix}\begin{pmatrix}
    x & y & z & t \\
    e & f & g & h
\end{pmatrix}=\begin{pmatrix}
    e & f & g & h \\
    x-l_0e & y-l_0f & z-l_0g & t-l_0h 
\end{pmatrix}.\]
If the output condition~\eqref{Eq: condition_bilinear} is not satisfied, it is again possible to use the partial quotients of $\alpha$ and $\beta$ by performing input transformations. In this case we have two possible different input transformations, because we can use either the partial quotients $a_i$ of the continued fraction of $\alpha$ or the partial quotients $b_i$ of the continued fraction of $\beta$. A natural choice in $\R$ is to alternate the input of one partial quotient of each continued fraction. After the input transformation of $\alpha$, the bilinear fractional transformation becomes
\[ \cfrac{x\left(a_0+\frac{1}{\alpha_1}\right)\beta+y\left(a_0+\frac{1}{\alpha_1}\right)+z\beta+t}{e\left(a_0+\frac{1}{\alpha_1}\right)\beta+f\left(a_0+\frac{1}{\alpha_1}\right)+g\beta+h}=\frac{(xa_0+z)\alpha_1\beta+(ya_0+t)\alpha_1+x\beta+y}{(ea_0+g)\alpha_1\beta+(fa_0+h)\alpha_1+e\beta+f},\]
In the following, we call it an \textit{$\alpha$-input transformation}.
After the input transformation of $\beta$, it becomes
\begin{equation*}
\frac{x\alpha\left(b_0+\frac{1}{\beta_1}\right)+y\alpha+z\left(b_0+\frac{1}{\beta_1}\right)+t}{e\alpha\left(b_0+\frac{1}{\beta_1}\right)+f\alpha+g\left(b_0+\frac{1}{\beta_1}\right)+h}=\frac{(xb_0+y)\alpha\beta_1+x\alpha+(zb_0+t)\beta_1+z}{(eb_0+f)\alpha\beta_1+e\alpha+(gb_0+h)\beta_1+g}.
\end{equation*}
In the following, we call it a \textit{$\beta$-input transformation}.
In matrix form, the two transformation can be represented as
\[\begin{pmatrix}
    x & y & z & t \\
    e & f & g & h
\end{pmatrix}\begin{pmatrix}
    a_0 & 0 & 1 & 0 \\
    0 & a_0 & 0 & 1 \\
    1 & 0 & 0 & 0 \\
    0 & 1 & 0 & 0 
\end{pmatrix}=\begin{pmatrix}
    xa_0+z & ya_0+t & x & y \\
    ea_0+g & fa_0+h & e & f 
\end{pmatrix},\]
for the $\alpha$-\textit{input transformation} and as
\[\begin{pmatrix}
    x & y & z & t \\
    e & f & g & h
\end{pmatrix}\begin{pmatrix}
    b_0 & 1 & 0 & 0 \\
    1 & 0 & 0 & 0 \\
    0 & 0 & b_0 & 1 \\
    0 & 0 & 1 & 0 
\end{pmatrix}=\begin{pmatrix}
    xb_0+y & x & zb_0+t & z \\
    eb_0+f & e & gb_0+h & g 
\end{pmatrix},\]
for the $\beta$-\textit{input transformation}.
Also in the case of the bilinear fractional transformation, it is possible to prove that, after a finite number of input transformations, condition~\eqref{Eq: condition_bilinear} is eventually satisfied, therefore $l_0=\lfloor\gamma\rfloor$ is computed correctly.

\subsection{Continued fractions in $\Q_p$}
In this section, we introduce the main algorithms for $p$-adic continued fractions. The idea is to use the usual Algorithm~\eqref{Alg: R} for simple continued fractions, by defining a suitable $p$-adic floor function $\lfloor\cdot\rfloor_p$. One of the first natural definitions for the floor function of a $p$-adic number has been provided by Ruban \cite{rub}.  Given $\alpha=\sum\limits_{n=-r}^{+\infty}c_np^n\in\mathbb{Q}_p$, with $c_n \in \{ 0, \ldots, p-1 \}$, Ruban's floor function is defined as \begin{equation}\label{Eq: rubfunc}
\lfloor\alpha\rfloor_p^\mathcal{R}=\sum\limits_{n=-r}^{0}c_np^n,
\end{equation}
and $\lfloor\alpha\rfloor_p^\mathcal{R}=0$ if $r<0$. 
Browkin's first algorithm, defined in \cite{browkin1978continued} and referred to here as \textit{Browkin I}, uses a similar floor function, but selects different representatives modulo $p$.
Given $\alpha = \sum\limits_{n=-r}^{+\infty} c_n p^n \in \mathbb{Q}_p$, where $c_n \in \left \{-\frac{p-1}{2}, \ldots, \frac{p-1}{2}\right \}$, Browkin's floor function, denoted by $s$, is defined as
\begin{equation}\label{Eq: sfunc}
s(\alpha)=\sum\limits_{n=-r}^{0}c_np^n,
\end{equation}
with $s(\alpha) = 0$ if $r < 0$. In \cite{browkin2001continued}, Browkin introduced another floor function to use in combination with the $s$ function. For $\alpha=\sum\limits_{n=-r}^{+\infty}c_np^n\in\mathbb{Q}_p$, where $c_n \in \{ -\frac{p-1}{2}, \ldots, \frac{p-1}{2} \}$, the floor function $t$ is defined as
\[t(\alpha)=\sum\limits_{n=-r}^{-1}c_np^n,\]
with $t(\alpha)=0$ if $r\leq 0$. Browkin's second algorithm, which we refer to as \textit{Browkin II}, works as follows: $\alpha_{0}=\alpha$ and then, for all $n\geq 0$,
\begin{align}\label{Alg: Br2} 
\begin{cases}
a_n=s(\alpha_n) &\text{if $n$ even}\\
a_n=t(\alpha_n) &\text{if $n$ odd and $v_p(\alpha_n-t(\alpha_n))= 0$} \\
a_n=t(\alpha_n)-sign(t(\alpha_n)) &\text{if $n$ odd and $v_p(\alpha_n-t(\alpha_n))\neq 0$}\\
\alpha_{n+1}=\frac{1}{\alpha_n-a_n}.
\end{cases}
\end{align}
Recently, a modification of \textit{Browkin II} has been proposed in \cite{murru2024new}, in order to improve its periodicity properties. For all $\alpha_0\in\mathbb{Q}_p$, the algorithm works as follows, for all $n\geq 0$

\begin{align}
\begin{cases}\label{Alg: MR}
a_n=s(\alpha_n) &\text{if $n$ even}\\
a_n=t(\alpha_n) &\text{if $n$ odd}\\
\alpha_{n+1}=\frac{1}{\alpha_n-a_n}.
\end{cases}
\end{align}

Throughout the paper, we mainly deal with the partial quotients of Ruban's and \textit{Browkin I} algorithms. In Section~\ref{Sec: otheralgo} also the results for Algorithm~\eqref{Alg: MR} are presented. Therefore, we provide the following definition.

\begin{Definition}\label{Def: represet}
Let $\mathcal{R}$, $\mathcal{B}$, and $\mathcal{T}$ denote the sets of all possible values of the floor functions $\lfloor \cdot \rfloor_p^\mathcal{R}$, $s$, and $t$, respectively. Therefore, these sets are:
\begin{align*}
\mathcal{R}&=\left\{\frac{c}{p^n } \ \Big| \ c, n\in\N, \ 0\leq c< p^{n+1} \right\}=\Z\left[\frac{1}{p}\right]\cap [0,p),\\
\mathcal{B}&=\left\{\frac{c}{p^n } \ \Big| \ n\in \N,\  c\in\Z, \ -\frac{p^{n+1}}{2}< c< \frac{p^{n+1}}{2} \right\}=\Z\left[\frac{1}{p}\right]\cap \left ( -\frac{p}{2} , \frac{p}{2} \right ),\\
\mathcal{T}&=\left\{\frac{c}{p^n } \ \Big| \ n\in \N,\  c\in\Z, \ -\frac{p^{n}}{2}< c< \frac{p^{n}}{2} \right\}=\Z\left[\frac{1}{p}\right]\cap \left ( -\frac{1}{2} , \frac{1}{2} \right ).
\end{align*}
\end{Definition}

It is well known that every finite continued fraction represents a rational number. The converse does not always hold in $\Q_p$ and it depends on the algorithm used to compute the partial quotients.
\begin{Proposition}[\cite{browkin1978continued,capuano2019effective,murru2024new}]\label{Prop: finexpansion}
Let $\alpha \in \Q$. Then:
\begin{enumerate}
\item The expansion of $\alpha$ obtained via \textit{Browkin I} algorithm is always finite.
\item The expansion of $\alpha$ obtained via Algorithm~\eqref{Alg: MR} is always finite.
\item If $\alpha < 0$, then for every prime number $p$, Ruban’s continued fraction of $\alpha$ does not terminate.
\item If $\alpha \ge 0$ and $\alpha \in \Z$, there are only finitely many prime numbers $p$ such that Ruban’s continued fraction of $\alpha$ does not terminate.
\item If $\alpha \ge 0$ and $\alpha \in \Q\backslash\Z$, there are only finitely many prime numbers $p$ such that Ruban’s continued fraction of $\alpha$ terminates.
\end{enumerate}
\end{Proposition}

\section{Metric theory of $p$-adic continued fractions}\label{Sec: measure}
In this section, we analyze the $p$-adic valuation of the partial quotients in the continued fraction expansions of elements of $\Q_p$ obtained via Ruban, \textit{Browkin I}, and Algorithm~\eqref{Alg: MR}. The metric results presented in this section are crucial for studying the termination of Algorithm~\ref{Alg: Gosp1} and Algorithm~\ref{Alg: Gosp2} for computing, respectively, the $p$-adic continued fraction of the M\"obius transformation and the bilinear fractional transformation.
In \cite{rub}, Ruban studies metric properties of $p$-adic numbers and $p$-adic continued fraction expansions. Let $\mu$ be the Haar measure on the additive group of $p$-adic numbers normed in such a way that $\mu(p \mathbb{Z}_p) = 1$. First, we report the main results of \cite{rub}.

\begin{Theorem}[\cite{rub}]\label{Thm: probval}
For any $i \geq 2$ and $y_j \in \mathcal{R}$, $j = 1,2,\ldots,i$, the sets 
\[\{ \alpha \in p \mathbb{Z}_p \mid a_j =y_j\}\]
are independent relative to $\mu$ and 
\[\mu\{ \alpha \in p \mathbb{Z}_p \mid a_j =y_j\} = p^{-2k},\]
with $k = -v_p(y_j)$.
\end{Theorem}

\begin{Theorem}[\cite{rub}]\label{Thm: frequency}
Let $y \in \mathcal{R}$, $v_p(y) = -k$. For \textit{almost all} $\alpha \in p \mathbb{Z}_p$ the frequency of repetition of $y$ in the decomposition of $\alpha$ into a continued fraction is the same, it is independent of $\alpha$ and is equal to $p^{-2k}$.
\end{Theorem}
Theorem~\ref{Thm: frequency} states that, for $\mu$-\textit{almost all} $p$-adic numbers, every partial quotient in $\mathcal{R}$ appears infinitely often in the Ruban $p$-adic continued fraction expansion. In particular, for any $v \in \Z$, there exist infinitely many indices $n$ such that $v_p(a_n)<v$.

\begin{Corollary}
Let $k\geq 1$. For \textit{almost all} $\alpha \in p \mathbb{Z}_p$ there exist infinitely many $i \in \mathbb{N}$ such that $v_p(a_i) \le -k$, where $a_i$ is the $i$-th partial quotient of the Ruban's continued fraction expansion of $\alpha$.
\end{Corollary}

The following corollary extends the previous result to all $p$-adic numbers.

\begin{Corollary}\label{Coro: unbounded}
Let $k\geq 1$. For $\mu$-\textit{almost all} $\alpha \in \Q_p$ there exist infinitely many $i \in \mathbb{N}$ such that $v_p(a_i) \le -k$, where $a_i$ is the $i$-th partial quotient of the Ruban's continued fraction expansion of $\alpha$.
\end{Corollary}
\begin{proof}
For $k\geq 1$, let 
\[U_k := \left\{ \alpha \in p\mathbb{Z}_p \,\middle|\, v_p(a_i) \ge -k \text{ for all but finitely many } i \ge 1 \right\}.\]By the previous corollary, $\mu(U_{k}) = 0$, and for all $a \in \mathcal{R}$, $\mu(a+U_{k}) = 0$, since the Haar measure is translation invariant.
Let 
\[V_k := \left\{ \alpha \in \Q_p \,\middle|\, v_p(a_i) \ge -k \text{ for all but finitely many } i \ge 1 \right\}.\]
We have $V_{k} = \bigcup\limits_{a \in \mathcal{R}} (a + U_{k})$ and, since $\mathcal{R} \subset \Q$ is countable, then $\mu(V_{k}) = 0$.
\end{proof}

Therefore, for any integer $k\geq 1$, the partial quotients of $\mu$-\textit{almost all} $\alpha \in \Q_p$ frequently have valuation less than $-k$.

\begin{Remark}
The results above hold also for \textit{Browkin I} algorithm. In particular, \cite[Theorem 3]{rub} holds also for the digits in $\{ -\frac{p-1}{2}, \ldots, \frac{p-1}{2} \}$. As a consequence, also Theorem~\ref{Thm: probval} holds for \textit{Browkin I} (see \cite[Theorem 4]{rub} for the proof). 
Finally, Theorem~\ref{Thm: frequency} (see \cite[Theorem 7]{rub} for the proof) also holds for the \textit{Browkin I} expansion, since the argument proceeds analogously in this setting and relies on the fact that both continued fraction expansions share the same approximation property,
\[\left| \alpha - \frac{A_n}{B_n} \right|_p < p^{-n} \quad \text{for all } n\ge 0,\]
where $\frac{A_n}{B_n} = [a_0, \ldots, a_n]$.
\end{Remark}

In the final part of this section, we study the $p$-adic valuation of the partial quotients obtained via Algorithm~\eqref{Alg: MR}. 
In this case, the partial quotients in even positions lie in $\mathcal{B}$, while those in odd positions lie in $\mathcal{T}$, as defined in Definition~\ref{Def: represet}. At the end of this section, we prove that for $\mu$-\textit{almost all} $\alpha \in \Q_p$ the partial quotients of the expansion obtained via Algorithm~\eqref{Alg: MR} in the odd positions have unbounded $p$-adic valuation. To do so, we use techniques similar to those of Ruban in \cite{rub}. Analogously to Theorem~\ref{Thm: probval}, the following result holds.

\begin{Theorem}\label{Thm: 4MR}
For arbitrary integral $i > 1$ and arbitrary $y_j$ ($j =1, 2, \ldots, i$) such that \begin{align*}y_{2j} \in \mathcal{B} \quad & \text{for} \quad 0<j \le \left \lfloor \frac{i}{2} \right \rfloor \\
y_{2j+1} \in \mathcal{T} \quad & \text{for} \quad 0\le j \le \left \lfloor \frac{i-1}{2} \right \rfloor
\end{align*} the sets $\{ \alpha \mid a_j =y_j\}$ in $p \Z_p$ are independent relative to $\mu$ and, if $-n_j = v_p(y_j)$, then
\[\begin{aligned}&\mu\{\alpha \mid a_j = y_j\} = p^{-2n_j +1} \quad &\text{if }j \equiv
 1 \pmod{2}, \\ 
&\mu\{\alpha \mid a_j = y_j\}= p^{-2n_j} \quad &\text{if }j \equiv 0 \pmod{2}.\end{aligned}\]
\end{Theorem}
\begin{proof}
The proof follows the same reasonings of  \cite[Theorem 4]{rub}.
\end{proof}

The following results are needed to prove the forthcoming Theorem~\ref{thm: 7}.

\begin{Proposition}\label{prop: cilindri}
Every $p$-ball in $p \Z_p$ is countable union of cylinders of the form
\[C(y_1, \ldots, y_n) := \{ \alpha \mid a_1=y_1, \ldots, a_n=y_n \},\]
with $y_{2i} \in \mathcal{B}$ for $0<i \le \left \lfloor \frac{n}{2} \right \rfloor$, $y_{2i+1} \in \mathcal{T}$ for $0\le i \le \left \lfloor \frac{n-1}{2} \right \rfloor$, and $\alpha = [0,a_1, a_2, \ldots]$. We call $n$ the order of the cylinder.
\end{Proposition}
\begin{proof}
Let $\alpha \in p\Z_p$, and let $\frac{A_n}{B_n}$ be the $n$-th convergent of the expansion of $\alpha$ via Algorithm~\eqref{Alg: MR}, i.e.
\[ \frac{A_n}{B_n} = [0,a_1, \ldots, a_n].\] 
We have
\[\left | \alpha - \frac{A_n}{B_n} \right |_p = p^{v_p(B_nB_{n+1})}.\]
Moreover, it can be proved by induction that
\[v_p(B_n) = v_p(a_1) + \cdots + v_p(a_n) \le -\left \lfloor \frac{n}{2} \right \rfloor,\]
and thus
\[v_p(B_nB_{n+1}) = 2 (v_p(a_1) + \cdots + v_p(a_n)) + v_p(a_{n+1}) \le -\left \lfloor \frac{n}{2} \right \rfloor - \left \lfloor \frac{n+1}{2} \right \rfloor = -n.\]
Therefore
\[\left | \alpha - \frac{A_n}{B_n} \right |_p \le p^{-n} \quad \text{and} \quad \left | \alpha - \frac{A_{n+1}}{B_{n+1}} \right |_p < p^{-n}.\]
Now we show that a cylinder of order $n+1$ is contained in a $p$-ball of radius $p^{-n}$.
Let
\[C=C(y_1, \ldots, y_{n+1}) = \{ \alpha \mid a_1=y_1, \ldots, a_{n+1}=y_{n+1} \}.\]
Let $x,y \in C$, then the $(n+1)$-th convergents relative to $x$ and $y$ are the same. Then,
\[\forall \ \alpha \in C(y_1, \ldots, y_{n+1}), \quad \left | \alpha - \frac{A_{n+1}}{B_{n+1}} \right |_p < p^{-n},\]
so that
\[C(y_1, \ldots, y_{n+1}) \subseteq B \left ( \frac{A_{n+1}}{B_{n+1}}, p^{-n} \right ).\]
We prove that every $p$-ball of radius $p^{-n}$ is countable union of cylinders of order $n+1$.
Let $B \subseteq p \Z_p$ a $p$-ball of radius $p^{-n}$. Given $\alpha \in B$, we consider the cylinder $C(a_1, \ldots, a_{n+1})$. We proved that
\[C(a_1, \ldots, a_{n+1}) \subseteq B \left ( \frac{A_{n+1}}{B_{n+1}}, p^{-n}\right ).\]
Since $\left | \alpha -  \frac{A_{n+1}}{B_{n+1}} \right |_p< p^{-n}$ and the $p$-ball $B$ has radius $p^{-n}$, it follows that $\frac{A_{n+1}}{B_{n+1}} \in B$. Consequently, $B=B\left ( \frac{A_{n+1}}{B_{n+1}}, p^{-n}\right )$.
Therefore, every $\alpha \in B$ belongs to a cylinder of order $n+1$ that is contained in $B$. Hence,
\[B = \bigcup_{C \subset B \text{ cylinder of order }n+1} C,\]
and therefore it is a countable union of cylinders, since the set of cylinders of order $n+1$ is countable.
\end{proof}

\begin{Corollary}\label{coro: sigmaMR}
Cylinders generate the Borel sigma-algebra of $p \Z_p$.
\end{Corollary}

\begin{Proposition}\label{Prop: MRmeas}
Let
\[
\begin{array}{rclcrcl}
T_t : p\Z_p &\longrightarrow& \Z_p, &\qquad& 
T_s : \Z_p &\longrightarrow& p\Z_p, \\[3pt]
x &\longmapsto& \displaystyle\frac{1}{x} - t\!\left(\frac{1}{x}\right), &&
x &\longmapsto& \displaystyle\frac{1}{x} - s\!\left(\frac{1}{x}\right),
\end{array}
\]
and let
\[T = T_s \circ T_t: p \Z_p \longrightarrow p \Z_p.\]
Then, the map $T$ is measure preserving (with respect to $\mu$). Notice that $T$ is not defined on all the elements of $p \Z_p$ with finite expansion (i.e. the rationals in $p \Z_p$), but the set of these elements has Haar measure $0$.
\end{Proposition}
\begin{proof}
The map $T$ can be expressed as follows
\begin{align*}
T : p\Z_p &\longrightarrow p\Z_p \\
 \alpha=[0,a_1,a_2, \ldots] &\longmapsto \alpha=[0,a_3,a_4, \ldots]
\end{align*}
where $[0,a_1,a_2, \ldots]$ is the continued fraction expansion of $\alpha$. 
By Proposition~\ref{prop: cilindri} and \cite[Theorem 5.7]{schilling2017measures} (Uniqueness of measures), it is sufficient to prove that $T$ preserves the measure of the cylinders. We prove that, given a cylinder $C(y_1, \ldots, y_n)$
\[\mu(C(y_1, \ldots, y_n)) = \mu(T^{-1}(C(y_1, \ldots, y_n))).\]
By Theorem~\ref{Thm: 4MR}
\[\mu(C(y_1, \ldots, y_n)) = \prod_{i=1}^{\left \lfloor \frac{n}{2} \right \rfloor} p^{-2k_{2i}} \prod_{i=1}^{\left \lfloor \frac{n-1}{2} \right \rfloor} p^{-2k_{2i+1}+1} = p^{-2 \sum\limits_{i=1}^n k_i + \left \lfloor \frac{n-1}{2} \right \rfloor},\]
where $v_p(y_i)=-k_i$ for $i=1, \ldots,n$. Moreover, by applying Theorem~\ref{Thm: 4MR} again, we have
\begin{align*}
\mu(T^{-1}(C(y_1, \ldots, y_n))) = \prod_{i=1}^n \mu \{\alpha \in p \Z_p \mid a_{i+2} = y_i \}=\prod_{i=1}^{\left \lfloor \frac{n}{2} \right \rfloor} p^{-2k_{2i}} \prod_{i=1}^{\left \lfloor \frac{n-1}{2} \right \rfloor} p^{-2k_{2i+1}+1}.   
\end{align*}
Hence, $\mu(C) = \mu(T^{-1}(C))$ for every cylinder, which concludes the proof.
\end{proof}

\begin{Theorem}\label{thm: 7}
Let $y \in \mathcal{T}$, with $v_p(y)=-k$, for $k\in\Z$. For \textit{almost all} $\alpha \in p \Z_p$, the frequency of repetition of $y$ in the odd positions of the continued fraction expansion of $\alpha$ via Algorithm~\eqref{Alg: MR} is the same, it is independent of $\alpha$, and equal to $p^{-2k+1}$.
\end{Theorem}
\begin{proof}
The technique of the proof follows similar steps as the proof of \cite[Theorem 7]{rub} for Ruban's continued fractions. Let $T: p \Z_p \to p \Z_p$ as in Proposition~\ref{Prop: MRmeas}. $T$ is measure preserving. It follows from Theorem~\ref{Thm: 4MR} that it is ergodic. Let $\alpha \in p \Z_p$ and let
\[f(\alpha) := \begin{cases}
1 &\text{if } a_1 = y \\
0 &\text{if } a_1 \ne y
\end{cases}.\]
Then, $f \in L_1(p \Z_p)$,
\[\int_{p\Z_p} f(x)\, d\mu =\mu \{\alpha \in p\Z_p \mid a_1=y \}= p^{-2k+1},\]
by Theorem~\ref{Thm: 4MR}, and the required result is obtained if we apply Birkhoff's Ergodic Theorem (\cite[Theorem 1.14]{walters2000introduction}) to the transformation $T$ and the function $f(\alpha)$. In this case, 
\[\lim\limits_{n \to \infty} \frac{1}{n} \sum_{i=0}^{n-1} f(T^i(\alpha)) = \int_{p\Z_p} f(\alpha)\, d\mu = p^{-2k+1},\]
for \textit{almost all} $\alpha \in p \Z_p$.
\end{proof}

\begin{Corollary}\label{Coro: MRvalunbounded}
Let $k \geq 1 $. For \textit{almost all} $\alpha \in \Q_p$ there exist infinitely many $i \in \mathbb{N}$ such that $v_p(a_{2i+1}) \le -k$, where $a_i$ is the $i$-th partial quotient of the Algorithm~\eqref{Alg: MR} $p$-adic continued fraction of $\alpha$.
\end{Corollary}
\begin{proof}
The proof follows the same steps as in Corollary~\ref{Coro: unbounded}.
\end{proof}

\section{Arithmetic of $p$-adic numbers}\label{Sec: aux}
Let $\alpha=\sum\limits_{n=-r}^{+\infty}c_np^n$ be a $p$-adic number, with $v_p(\alpha)=-r$ and $c_n\in \Z/p\Z$ for all $n\ge -r$. Let
\[\lfloor\alpha\rfloor_p=\sum\limits_{n=-r}^{0}c_np^n,\]
so that $\lfloor\alpha\rfloor_p=0$ for $r<0$. 
The floor function considered is either \eqref{Eq: rubfunc} of Ruban's algorithm or \eqref{Eq: sfunc} of \textit{Browkin I} algorithm, depending on the chosen representatives. The results of the following sections hold either choosing the partial quotients inside the set $\mathcal{R}$ or the set $\mathcal{B}$. In Section \ref{Sec: otheralgo} we deal also with the partial quotients in $\mathcal{T}$. By the first $k$ \textit{digits} of $\alpha=\sum\limits_{n=-r}^{+\infty}c_np^n$ we mean the coefficients $c_n$ with $n=-r, \ldots, -r+k-1$ and $c_{-r} \ne 0$.\\

In the following sections, we aim to compute the floor function of the M\"obius transformation and the bilinear fractional transformation of $p$-adic continued fractions, using only the knowledge of their partial quotients.
In order to compute correctly the floor function of a transformation $\gamma=\sum\limits_{n=-r_\gamma}^{+\infty}l_np^n$, we must be able to compute the digits $l_{-r_{\gamma}},\ldots,l_0$. Here we list some auxiliary results for the arithmetic manipulation of the $p$-adic digits by simple operations. In the next lemma, we determine the number of digits that is possible to determine for the inverse of a $p$-adic number.

\begin{Lemma}\label{Lem: inverse}
Let $\alpha=\sum\limits_{n=-r}^{+\infty}c_np^n$, with $r\in\Z$ and $c_{-r} \ne 0$, be a nonzero $p$-adic number and let $\frac{1}{\alpha}=\sum\limits_{n=r}^{+\infty}d_np^n$ be its inverse. Then, for all $k\geq 0$, the coefficients $d_r,\ldots,d_{r+k-1}$ are uniquely determined by the coefficients $c_{-r},\ldots,c_{-r+k-1}$.
\begin{proof}
We fix $k\geq 0$. The following holds
\begin{equation}\label{Eq: inverseprod}1=\alpha\frac{1}{\alpha}=(c_{-r}p^{-r}+c_{-r+1}p^{-r+1}+\cdots)(d_{r}p^r+d_{r+1}p^{r+1}+\cdots).\end{equation}
In order to compute the $k$ coefficients $d_r,\ldots,d_{r+k-1}$ of $\frac{1}{\alpha}$, we solve the linear system
\begin{align}\label{Eq: dsystem}
\begin{cases}
c_{-r}d_r=1\\
c_{-r}d_{r+1}+c_{-r+1}d_r=0\\
c_{-r}d_{r+2}+c_{-r+1}d_{r+1}+c_{-r+2}d_{r}=0\\
\quad \vdots \\
\sum\limits_{i=0}^{k-1} c_{-r+i}d_{r+k-1-i}=0,
\end{cases}
\end{align}
where the equations are in $\F_p$ and describe the digits of the terms $1,p,\ldots,p^{k-1}$ of the product \eqref{Eq: inverseprod}. Let us consider the $k\times k$ matrix of the coefficients

\[C=\begin{pmatrix}
\ c_{-r} &   0 & 0 & \ldots &  0 \  \\
\ c_{-r+1} & c_{-r} & 0  & \ldots & 0 \  \\
\ c_{-r+2} & c_{-r+1} & c_{-r} & \ddots  & \vdots  \  \\
\ \vdots & \vdots & \ddots & \ddots & 0  \ \\
\ c_{-r+k-1} & c_{-r+k-2} & \ldots & c_{-r+1} & c_{-r} \ \\
\end{pmatrix}.\]
Notice that $\det(C)\neq 0$ as $c_{-r}$ is nonzero modulo $p$. Therefore, there is a unique solution $(d_{r},\ldots,d_{r+k-1})$ to \eqref{Eq: dsystem}, determined by $c_{-r},\ldots,c_{-r+k-1}$, and the thesis follows.
\end{proof}
\end{Lemma}

In the next two results, we compute the number of digits that we are able to determine for, respectively, the product and the sum of two $p$-adic numbers.

\begin{Lemma}\label{Lem: prodab}
Let us consider $\alpha=\sum\limits_{n=-r_\alpha}^{+\infty}c_np^n$ and $\beta=\sum\limits_{n=-r_\beta}^{+\infty}d_np^n$. Suppose we know the first $h_\alpha$ digits $c_{-r_\alpha},\ldots,c_{-r_\alpha+h_\alpha-1}$ of $\alpha$ and the first $h_\beta$ digits $d_{-r_\beta},\ldots,d_{-r_\beta+h_\beta-1}$ of $\beta$. Then, it is possible to uniquely determine the first $k=\min\{h_\alpha,h_\beta\}$ digits of 
\[\alpha\beta=\sum\limits_{n=-r_\alpha-r_\beta}^{+\infty}e_np^n,\] but not the $(k+1)$-th digit.
\begin{proof}
Notice that, for all $t\geq0$, the digits of $\alpha\beta$ are
\begin{equation}\label{Eq: et}
e_{-r_\alpha-r_\beta+t}=\sum\limits_{i=0}^{t} c_{-r_\alpha+i}d_{-r_\beta+t-i},
\end{equation}
where, for $t=0,\ldots,k-1$, all the coefficients involved are known by hypothesis. For $t=k$, either $c_{-r_\alpha+h_\alpha}$ or $d_{-r_\beta+h_\beta}$ appears in \eqref{Eq: et}, so it is not possible to determine $e_{-r_\alpha-r_\beta+k}$.
\end{proof}
\end{Lemma}

\begin{Lemma}\label{Lem: sumab}
Let us consider $\alpha=\sum\limits_{n=-r_\alpha}^{+\infty}c_np^n$ and $\beta=\sum\limits_{n=-r_\beta}^{+\infty}d_np^n$. Let us suppose to know the first $h_\alpha$ digits $c_{-r_\alpha},\ldots,c_{-r_\alpha+h_\alpha-1}$ of $\alpha$ and the first $h_\beta$ digits $d_{-r_\beta},\ldots,d_{-r_\beta+h_\beta-1}$ of $\beta$. Then it is possible to uniquely determine $k$ digits, but not $k+1$, of
\[\alpha+\beta=\sum\limits_{n=-r}^{+\infty}e_np^n,\] where $-r = v_p(\alpha + \beta)$ and
\[k= \min \{ -r_\alpha + h_\alpha
, -r_\beta + h_\beta  \} - v_p\left(\sum\limits_{n=-r_\alpha}^{-r_\alpha+h_\alpha-1}c_np^n + \sum\limits_{n=-r_\beta}^{-r_\beta+h_\beta-1}d_np^n\right).\]
If $k\leq 0$, we are not able to determine any $p$-adic digit of $\alpha+\beta$.
\begin{proof}
Let us consider, without losing generality, that $ -r_\alpha + h_\alpha - 1 \le -r_\beta + h_\beta - 1$. If $v_p(\alpha + \beta) \leq -r_\alpha + h_\alpha - 1$, then we know the digits of $\alpha$ and $\beta$ relative to the powers $p^i$ for $i \le -r_\alpha + h_\alpha - 1$, but we do not know $c_{-r_\alpha + h_\alpha}$. Hence, we can determine the digits of $\alpha + \beta$ relative to the powers $p^i$ with $i= v_p(\alpha + \beta), \ldots, -r_\alpha + h_\alpha - 1$. This is a total of
\[k=\min \{ -r_\alpha + h_\alpha
-1, -r_\beta + h_\beta -1 \}-v_p(\alpha+\beta)+1,\]
digits, and the thesis follows for this case.
If $v_p(\alpha + \beta) > -r_\alpha + h_\alpha - 1$, then we are not able to determine any digit of $\alpha + \beta$.
\end{proof}
\end{Lemma}

\section{M\"obius transformation of $p$-adic continued fractions}\label{Sec: Mobpadic}
In this section, we study how to compute Ruban's and \textit{Browkin I} $p$-adic continued fraction expansion of the M\"obius transformation of $\alpha\in\Q_p$, given by
\[\gamma=\frac{x\alpha+y}{z\alpha+t},\]
with $x,y,z,t\in\Z$. Note that if
\[\det\begin{pmatrix}
    x & y \\
    z & t
\end{pmatrix}=xt-yz=0,\]
then $\gamma$ is a rational number, and we know how to compute its continued fraction. Therefore, we assume in the following that $xt-yz\neq0$. First of all, we analyze the conditions under which the $p$-adic floor function of such a transformation can be determined. In particular, we identify when all the coefficients of $\left\lfloor\frac{x\alpha+y}{z\alpha+t}\right\rfloor_p$ can be recovered from the knowledge of a finite number of partial quotients $a_i$ of the $p$-adic continued fraction of $\alpha$. The main result of this analysis is presented in Theorem~\ref{thmfloor}. Before stating it, we compute how many $p$-adic digits of $x\alpha + y$ can be determined from the knowledge of a given number of $p$-adic digits of $\alpha$.

\begin{Remark}\label{Rem: pathological} In general, $v_p(x\alpha+y)$ can be greater than $v_p(x\lfloor\alpha\rfloor_p+y)$. For example, let us consider,
\[\alpha=1+4\cdot5+5^2+\cdots\in\Q_5,\]
with $x=1$ and $y=4$. Then,
\[v_5(x\lfloor\alpha\rfloor_5+y)=v_5(1+4)=1,\]
but
\[v_5(x\alpha+y)=v_5(2\cdot5^2+\cdots)=2,\]
hence the valuation of $x\alpha+y$ cannot be, in general, determined by the valuation of $x\lfloor\alpha\rfloor_5+y$. Note that this can occur only when $v_p(x\lfloor\alpha\rfloor_p+y) \ge v_p(x) + 1$.
\end{Remark}

\begin{Proposition}\label{Pro: digitsum}
Let $x,y\in\Q$ and let $\alpha=\sum\limits_{n=-r}^{+\infty}c_np^n$. Let us suppose to know the first $h$ digits $c_{-r},\ldots,c_{-r+h-1}$ of $\alpha$. Then we can uniquely determine the first
\[k=v_p(x)-r+h-v_p\left(x\sum\limits_{n=-r}^{-r+h-1}c_np^n+y\right)\]
digits of $x\alpha+y=\sum\limits_{n=v_p(x\alpha+y)}^{+\infty}d_np^n.$
When $k\leq0$, that is, when 
\[v_p(x)-r+h\leq v_p\left(x\sum\limits_{n=-r}^{-r+h-1}c_np^n+y\right),\]  then no $p$-adic digit of $x\alpha+y$ can be determined from the knowledge of $h$ digits of $\alpha$.
\begin{proof}
By hypothesis, we know all digits of $x$ and $y$, and so, by Lemma~\ref{Lem: prodab}, we know the first $h$ digits of $x \alpha$. The thesis follows applying Lemma~\ref{Lem: sumab} with $x \alpha$ and $y$. In particular, in the notation of Lemma~\ref{Lem: sumab}, since we know all digits of $y$, $h_y= \infty$, and $h_{x \alpha} = h$, so that
\[\begin{aligned}k &= \min \{ v_p(x \alpha) + h , v_p(y) + \infty  \} -v_p\left(x\sum\limits_{n=-r}^{-r+h-1}c_np^n+y\right) \\
&= v_p(x) -r + h  -v_p\left(x\sum\limits_{n=-r}^{-r+h-1}c_np^n+y\right).
\end{aligned}\]
This proves the claim.
\end{proof}
\end{Proposition}

\begin{Remark}\label{Rem: riporto}
In general, knowing $k$ digits of $\alpha$ is not sufficient for computing any digit of $x\alpha+y$. This occurs when 
\[v_p(x\alpha)=v_p(y) \quad  \text{and}\quad  v_p(x\alpha+y)\geq v_p(x\alpha)+k.\]
Indeed, in such a case, some of the unknown digits of $\alpha$ are necessary to determine the first digit of $x\alpha + y$. For example, let $p=5$, $(x,y)=(5,8)$, and suppose we know the first three digits of the $5$-adic expansion of
\[\alpha = \frac{2}{5} + 3 + 4 \cdot 5 + \delta,\]
where $\delta \ne 0$ is a $5$-adic number with $v_5(\delta) \ge 2$. In this case, $v_5(\alpha) = -1$ and
\[x\alpha + y = 5^3 + 5 \cdot \delta.\]
Since $v_5(5 \cdot \delta) \ge 3$ and the digits of $5 \cdot \delta$ are unknown, the $5$-adic digits of $x\alpha + y$ cannot be computed.
\end{Remark}

\begin{Remark}\label{Rem: mostpato}
By Proposition~\ref{Pro: digitsum}, the case when
\[v_p(x\lfloor\alpha\rfloor_p+y) \ge v_p(x) + 1,\]
is the most pathological. In fact, in this case, we are not able to determine any digit of the $p$-adic series of $x\alpha+y$. 
Thus, when this occurs for either the numerator $x\alpha+y$ or the denominator $z\alpha+t$, the knowledge of \(\lfloor \alpha \rfloor_p\) is not sufficient to compute any digit of
\[
\gamma = \frac{x\alpha + y}{z\alpha + t},
\]
and, in particular, the floor function \(\lfloor \gamma \rfloor_p\) cannot be obtained from it.
This is the most undesired situation, that we are going to handle in Section~\ref{Sec: MobInp}.

\end{Remark}

In the following remark, we note that we can assume, without loss of generality, $v_p(\alpha) \le 0$ when computing the $p$-adic floor function of $\frac{x\alpha+y}{z\alpha+t}$.

\begin{Remark}
If $v_p(\alpha)\geq1$, then $\lfloor\alpha\rfloor_p=0$ and we do not know any of its digits. However, after an input transformation, we obtain
\[\frac{(x\lfloor\alpha\rfloor_p+y)\alpha_1+x}{(z\lfloor\alpha\rfloor_p+t)\alpha_1+z}=\frac{y\alpha_1+x}{t\alpha_1+z},\]
where now
\[v_p(\alpha_1)=v_p\left(\frac{1}{\alpha}\right)<0.\]
Therefore, the hypothesis $v_p(\alpha)\leq 0$ in the forthcoming Theorem~\ref{thmfloor} is not a restriction.
\end{Remark}

The following theorem provides the necessary and sufficient conditions for the determination of the floor function
\[\lfloor\gamma\rfloor_p=\left\lfloor\frac{x\alpha+y}{z\alpha+t}\right\rfloor_p,\]
given the knowledge of $x,y,z,t\in\Q$ and $\lfloor\alpha\rfloor_p$.

\begin{Theorem}\label{thmfloor}
Let $\alpha\in\Q_p$ and $x,y,z,t\in\Q$, with $xt - yz \ne 0$. Let $v_p(\alpha)=-r\leq0$ and suppose that $\lfloor\alpha\rfloor_p=\sum\limits_{n=-r}^{0}c_np^n$ is known. 
If $\lfloor\alpha\rfloor_p=\alpha$, then we are always able to compute  $\left\lfloor\frac{x\alpha+y}{z\alpha+t}\right\rfloor_p$. Suppose $\lfloor\alpha\rfloor_p \ne \alpha$.
If either \[v_p(x\lfloor\alpha\rfloor_p+y)\geq v_p(x)+1,\] 
or
\[v_p(z\lfloor\alpha\rfloor_p+t)\geq v_p(z)+1,\]
then it is not possible to determine $\left\lfloor\frac{x\alpha+y}{z\alpha+t}\right\rfloor_p$. In the other cases, let us denote
\begin{equation*}
\begin{aligned}
v^{(1)}&=v_p(x\lfloor\alpha\rfloor_p+y),\\
v^{(2)}&=v_p(z\lfloor\alpha\rfloor_p+t),
\end{aligned}
\end{equation*}
and let
\[m=\min\{v_p(x)-v^{(1)},v_p(z)-v^{(2)}\}.\]
Then, the $p$-adic floor function $\left\lfloor\frac{x\alpha+y}{z\alpha+t}\right\rfloor_p$ is uniquely determined by $\lfloor\alpha\rfloor_p$ if and only if $m\geq v^{(2)}-v^{(1)}$. In particular, in this case,
\[\left\lfloor\frac{x\alpha+y}{z\alpha+t}\right\rfloor_p=\left\lfloor\frac{x \lfloor \alpha\rfloor_p+y}{z \lfloor \alpha\rfloor_p+t}\right\rfloor_p.\]

\begin{proof}
If $\lfloor\alpha\rfloor_p=\alpha$, then we know exactly $\frac{x\alpha+y}{z\alpha+t}$ and so we can determine $\left\lfloor\frac{x\alpha+y}{z\alpha+t}\right\rfloor_p$.
Suppose now that $\alpha \ne \lfloor \alpha \rfloor_p$. By hypothesis, we know $\lfloor\alpha\rfloor_p=\sum\limits_{n=-r}^{0}c_np^n$, so that we are able to compute the first $r+1$ digits of the $p$-adic expansion of $\alpha$. If either $v_p(x\lfloor\alpha\rfloor_p+y)\geq v_p(x)+1$ or $v_p(z\lfloor\alpha\rfloor_p+t)\geq v_p(z)+1$, then, by Proposition~\ref{Pro: digitsum}, no digits of either $x\alpha+y$ or $z\alpha+t$, can be computed; hence, it is not possible to determine $\left\lfloor\frac{x\alpha+y}{z\alpha+t}\right\rfloor_p$. If this is not the case, that is if
\begin{align*}
v_p(x\lfloor\alpha\rfloor_p+y)&\leq v_p(x),\\
v_p(z\lfloor\alpha\rfloor_p+t)&\leq v_p(z),
\end{align*}
then, by Remark~\ref{Rem: pathological}, we have:
\begin{align*}
    v_p(x\alpha+y)&=v_p(x\lfloor\alpha\rfloor_p+y),\\
    v_p(z\alpha+t)&=v_p(z\lfloor\alpha\rfloor_p+t).
\end{align*}
By Proposition~\ref{Pro: digitsum}, we are able to compute
\[k_1=v_p(x)-v_p(x\alpha+y)+1\]
digits of the $p$-adic expansion of $x\alpha+y$, and
\[k_2=v_p(z)-v_p(z\alpha+t)+1\]
digits of the $p$-adic expansion of $z\alpha+t$.
Using Lemma~\ref{Lem: inverse}, we can compute also the first $k_2$ digits of $\frac{1}{z\alpha+t}$.
Notice that, in any case, we know $x\alpha+y$ up to the digit of $p^{v_p(x)}$ and $z\alpha+t$ up to the digit of $p^{v_p(z)}$.
Hence,
\[x\alpha+y=\sum\limits_{n=v^{(1)}}^{+\infty}l_np^n, \quad \frac{1}{z\alpha+t}=\sum\limits_{n=-v^{(2)}}^{+\infty}m_np^n,\]
for which the first $v_p(x) - v^{(1)} + 1$ and $v_p(z) - v^{(2)} + 1$ digits can be computed, respectively.
Therefore, we can write
\[\frac{x\alpha+y}{z\alpha+t}=\left(\sum\limits_{n=v^{(1)}}^{+\infty}l_np^n\right)\left(\sum\limits_{n=-v^{(2)}}^{+\infty}m_np^n\right)=\sum\limits_{n=v^{(1)}-v^{(2)}}^{+\infty}s_np^n,\]
and by Lemma~\ref{Lem: prodab}, we know the first $m+1$ digits $s_{v^{(1)}-v^{(2)}},\ldots,s_{v^{(1)}-v^{(2)}+m}$, with
\[m=\min\{v_p(x)-v^{(1)},v_p(z)-v^{(2)}\}.\]
It turns out that we can compute
\[\left\lfloor\frac{x\alpha+y}{z\alpha+t}\right\rfloor_p=\sum\limits_{n=v^{(1)}-v^{(2)}}^{0}s_np^n,\]
if and only if $v^{(1)}-v^{(2)}+m\geq 0$, i.e. if and only if 
\begin{equation*}m\geq v^{(2)}-v^{(1)}.
\end{equation*}
In this case, all the computations involve only the first $r+1$ digits of $\alpha$, i.e. the digits of its floor function $\lfloor\alpha\rfloor_p=\sum\limits_{n=-r}^{0}c_np^n$, that is known. Therefore,
\[\left\lfloor\frac{x\alpha+y}{z\alpha+t}\right\rfloor_p=\left\lfloor\frac{x \lfloor \alpha\rfloor_p+y}{z \lfloor \alpha\rfloor_p+t}\right\rfloor_p,\]
and the statement is proved.
\end{proof}
\end{Theorem}

When the conditions of Theorem~\ref{thmfloor} are not initially satisfied and the $p$-adic floor function cannot be computed directly, the desired conditions may sometimes be achieved by performing the input transformation using further partial quotients of $\alpha$. This approach is described in the next section, where we analyze the behavior of $\frac{x\alpha+y}{z\alpha+t}$ and its floor function after the input transformations.

\subsection{Input transformation}\label{Sec: MobInp}
In the spirit of classical Gosper's algorithm presented in Section~\ref{Sec: Mobius}, we would like to use the partial quotients of $\alpha$ by performing input transformations until the hypotheses of Theorem~\ref{thmfloor} are satisfied, hence allowing to compute the partial quotient of the transformation.
In this section, we analyze how using more partial quotients of $\alpha$ can help in computing the floor function of a M\"obius transformation. However, we are going to see that, unlike the classical Gosper's algorithm in $\mathbb{R}$, it is not guaranteed in general that the hypotheses of Theorem~\ref{thmfloor} are eventually satisfied. This means that there exist continued fractions $\alpha=[a_0,a_1,\ldots]$ and $x,y,z,t\in\Q$ for which it is not possible to determine the partial quotients of the continued fraction expansion of 
\[\gamma=\frac{x\alpha+y}{z\alpha+t},\]
using only a finite number of partial quotients of $\alpha$.
However, using the result from measure theory of Section~\ref{Sec: measure}, we prove that the set of $\alpha$ for which the floor function cannot be determined has Haar measure $0$.

The input transformation is identical to that in Equation~\eqref{Eq: inputtransformationsec2}, as it involves only arithmetic manipulations. 
For notational convenience, the following definition introduces the recurrences satisfied by the coefficients of the M\"obius transformation after a given number of input transformations.
\begin{Definition}\label{Def: recur}
Let us consider the continued fraction $[a_0,a_1,\ldots]$ and $x_0,y_0,z_0,t_0\in \Q$. Let us define, the sequences $\{x_n\}_{n\geq 0}$, $\{y_n\}_{n\geq 0}$, $\{z_n\}_{n\geq 0}$, $\{t_n\}_{n\geq 0}$ as:
\[\begin{cases}
x_{n+1}=x_na_n+y_n\\
y_{n+1}=x_n\\
z_{n+1}=z_na_n+t_n\\
t_{n+1}=z_n.
\end{cases}\]
\end{Definition}
It is not hard to see that, in the notation of Definition~\ref{Def: recur}, the following equality holds, for all $n\geq 1$:
\[\frac{x_0 \alpha + y_0}{z_0 \alpha + t_0} = \frac{x_{n} \alpha_{n} + y_{n}}{z_{n} \alpha_{n} + t_{n}}.\]

In Proposition~\ref{Pro: digitsum} and Remark~\ref{Rem: mostpato}, we have seen that if
\begin{equation}\label{Eq: unhoped}
v_p(x\lfloor\alpha\rfloor_p+y)\geq v_p(x)+1,
\end{equation}
then all the known digits are canceled due to carry-overs in the $p$-adic expansion of $x\alpha+y$. The following lemma studies the behavior of $v_p(x\alpha+y)$ after the input transformations and it is heavily used in the following arguments.

\begin{Lemma}\label{Lem: rec}
Let $\{a_n\}_{n\geq 0}$ be a sequence of $p$-adic numbers such that $v_p(a_n)=-r_n$, with $r_n\geq 1$ for all $n\geq 1$. Let $x_0,y_0\in\Q_p$. Let us denote
\[\mu_n=\min\{ v_p(x_na_n),v_p(y_n)\},\]
where the sequences $\{x_n\}_{n\geq 0}$ and $\{y_n\}_{n\geq 0}$ are as in Definition~\ref{Def: recur}.
If, for some $n\geq 0$,
\[\mu_n\leq v_p(x_na_n+y_n)\leq \mu_n+r_n,\]
then \[v_p(x_ma_m)< v_p(y_m) \quad \text{for all} \quad m\geq n+1.\]

\begin{proof}
First of all notice that, if $v_p(x_na_n+y_n)=\mu_n$, then
\[v_p(a_{n+1}x_{n+1}) = v_p(a_{n+1}) + \mu_n \le v_p(a_{n+1}) + v_p(a_{n}) + v_p(x_{n})<v_p(x_{n})=v_p(y_{n+1}).\]
If $\mu_n+1\leq v_p(x_na_n+y_n)\leq \mu_n+r_n$,
then some carry-over occurs in the lower-order digits. In this case we have:
\begin{align*}
v_p(x_{n+1}a_{n+1})\leq \mu_n+r_n-r_{n+1}=v_p(x_n)-r_n+r_n-r_{n+1}<v_p(x_n)=v_p(y_{n+1}),
\end{align*}
and the claim is proved.
\end{proof}
\end{Lemma}

Lemma~\ref{Lem: rec} states that if we are not in the case where
\[v_p(x\alpha+y)\geq v_p(x)+1,\]
then from the next input transformation onward we remain in a \emph{good} case, that is the one for which
\[v_p(x\alpha+y)=\min\{v_p(x\alpha),v_p(y)\},\]
and no digit carry-overs occur. In particular, we end up in the specific case where we have $v_p(x\alpha)<v_p(y)$.
On the contrary, we show in the following remark that, whenever we are in the \textit{pathological} case when the hypotheses of Lemma~\ref{Lem: rec} are not satisfied, any of the cases can occur after one input transformation.

\begin{Remark}\label{Rem: fromdisastroso}
In the notation of Lemma~\ref{Lem: rec}, if $v_p(x_na_n+y_n) > \mu_n+r_n$, then after an input transformation, we may either remain in this case or end up in a \textit{good} case. 

For example, with Ruban's algorithm, let $p\ge3$ and let
\[x_0 = p^2 + p^3 + p^4, \quad y_0 = (p-1) \cdot p + (p-2) \cdot p^2, \quad \text{and} \quad a_0 = \frac{1}{p}.\]
We have 
\[v_p(a_0x_0 + y_0) = v_p(2 \cdot p^3)=3 > 2 = \mu_0 + r_0.\]
After an input transformation, \[x_1 = a_0x_0 + y_0 = 2 \cdot p^3, \quad y_1 = p^2 + p^3 + p^4.\]
If $a_1 = \frac{(p-1)/2}{p} + (p-1)$ then 
\[v_p (a_1 x_1 + y_1) = v_p(3 \cdot p^4) = 4 > 3 = \mu_1 + r_1,\] and hence we remain in the same case.
If $v_p(a_1) \ne -1$, then $v_p(a_1 x_1) \ne v_p(y_1)$ and $v_p(x_{1}a_{1}+y_{1})= \mu_{1}$, so that we are in a good case.
\end{Remark}

Before delving into the case
\[v_p(x\alpha+y)\geq v_p(x)+1,\]
we prove that, in all other cases, the value of $v^{(2)}-v^{(1)}$ appearing in Theorem~\ref{thmfloor} becomes fixed after a certain number of input transformations and remains constant thereafter.
\begin{Proposition}\label{Prop: othername}
Consider a transformation $\frac{x \alpha + y}{z \alpha + t}$ such that
\begin{align*}
 v_p(x\alpha+y)\le v_p(x) \quad \text{and} \quad v_p(z\alpha+t)\le v_p(z).
\end{align*}
Then, for all $n\geq 1$,
\begin{align*}v_p(x_na_n)<v_p(y_n) \quad \text{and} \quad v_p(z_na_n)<v_p(t_n).
\end{align*}
Moreover, the quantity \[\left\lfloor \frac{x \alpha + y}{z \alpha + t}\right\rfloor_p = \left\lfloor\frac{x_n\alpha_n+y_n}{z_n\alpha_n+t_n}\right\rfloor_p,\]
can be determined for some \(n \ge 1\) if and only if 
\begin{equation}\label{Eq: z1x1}
r_n=-v_p(\alpha_n)\geq v_p(z_1)-v_p(x_1).
\end{equation}
\end{Proposition}
\begin{proof}
The hypotheses of Lemma~\ref{Lem: rec} are satisfied, so that, for all $n\geq 1$,
\[v_p(x_na_n)<v_p(y_n), \quad v_p(z_na_n)<v_p(t_n).\]
This means that, using the notation of Theorem~\ref{thmfloor}, and writing each quantity for the generic $\alpha_n$, for all $n\geq 1$,
\begin{align*}
v^{(1)}_n&=\min \{v_p(x_n)-r_n,v_p(y_n)\}=v_p(x_n)-r_n,\\
v^{(2)}_n&=\min \{v_p(z_n)-r_n,v_p(t_n)\}=v_p(z_n)-r_n.
\end{align*}
Therefore, for all $n\geq 1$, 
\[v^{(2)}_n-v_n^{(1)}=v_p(z_n)-v_p(x_n).\] In particular, for $n\geq 2$,
\[v^{(2)}_n-v_n^{(1)}=v_p(z_n)-v_p(x_n)=v_p(a_{n-1}z_{n-1})-v_p(a_{n-1}x_{n-1})=v_p(z_{n-1})-v_p(x_{n-1}),\]
where we have used Lemma~\ref{Lem: rec} in the second equality. Therefore, by proceeding inductively we have, for all $n\geq 1$,
\[v^{(2)}_n-v_n^{(1)}=v_p(z_1)-v_p(x_1),\]
where 
\begin{align*}
v_p(x_1) &= v^{(1)} = \min \{ v_p(x \alpha), v_p(y) \},\\
\quad v_p(z_1) &=  v^{(2)} = \min \{ v_p(z \alpha), v_p(t) \}.
\end{align*}

Therefore, by the condition of Theorem~\ref{thmfloor} we are able to determine \[\left\lfloor\frac{x_n\alpha_n+y_n}{z_n\alpha_n+t_n}\right\rfloor_p,\]
for some $n\geq1$ if and only if 
\[r_n=-v_p(\alpha_n)\geq v_p(z_1)-v_p(x_1),\]
and the claim is proved.
\end{proof}

The importance of Proposition~\ref{Prop: othername} is that, once we are in a \emph{good} case, we know that $v_n^{(2)}-v_n^{(1)}$ is constant for all $n\ge1$. Therefore, the problem of computing the partial quotient
\[\lfloor\gamma\rfloor_p=\left\lfloor\frac{x\alpha+y}{z\alpha+t}\right\rfloor_p,\]
depends on the existence of an index $n\ge0$ for which $\alpha_n$ has sufficiently small valuation.
For this reason, as $v_p(x_1)$ and $v_p(z_1)$ are fixed, it is not guaranteed that condition~\eqref{Eq: z1x1} is ever satisfied for some $n\geq 0$. In fact, it heavily depends on the valuation of the partial quotients of $\alpha$.

\begin{Example}\label{Exa: neveroutput}
Let $p=5$, $\alpha \in \Q_5$ and
\[\alpha = \left [1, \overline{\frac{1}{5}}\right ].\]
Suppose we want to compute the continued fraction of the transformation $\gamma=\frac{\alpha}{25}$.
Using the above notation, $(x,y,z,t)=(1,0,0,25)$. Then $v^{(1)} =  0$, $v^{(2)} = 2$ and
\[m = \min \{v_5(x) - v^{(1)}, v_5(z) - v^{(2)} \} = 0,\]
so that $m < v^{(2)} - v^{(1)}$. After an input transformation, we obtain $(x_1,y_1,z_1,t_1)=(1,1,25,0)$. Since
\[1=r_n = - v_5(\alpha_n) < 2 = v_5(z_1)-v_5(x_1) \quad \text{for all }n\ge1,\]
by Proposition~\ref{Prop: othername}, we are never able to compute $\left \lfloor \frac{\alpha}{25}\right \rfloor_5$.
\end{Example}

We have seen in Example~\ref{Exa: neveroutput} that, even if we are guaranteed that 
\begin{align*}
 v_p(x_n\alpha_n+y_n)&\le v_p(x_n),\\
 v_p(z_n\alpha_n+t_n)&\le v_p(z_n),
\end{align*}
it can still happen that we are never able to perform an output, depending on the valuation of the partial quotients.\bigskip

In what follows, and for the remainder of this section, we handle the last problematic case when inequality~\eqref{Eq: unhoped} holds either for the numerator or the denominator of $\frac{x\alpha+y}{z\alpha+t}$. We have seen in Remark~\ref{Rem: fromdisastroso} that if we are in the case 
\[v_p(x_n a_n + y_n) \geq v_p(x_n)+1,\]
then, after one input transformation, we may obtain again 
\[v_p(x_{n+1} a_{n+1} + y_{n+1}) \geq v_p(x_{n+1})+1,\]
hence remaining in the case where we are unable to determine any $p$-adic digit of the transformation
\[\frac{x_{n+1}\alpha_{n+1}+y_{n+1}}{z_{n+1}\alpha_{n+1}+t_{n+1}}.\]
The remaining question is whether, after finitely many input transformations, we are guaranteed to reach a \textit{good} case, i.e. whether there exists an index $m$ such that
\[v_p(x_{m} a_{m} + y_{m}) \leq v_p(x_{m}),\]
or whether it is instead possible that 
\begin{equation}\label{Eq: patoldis}
v_p(x_{n} a_{n} + y_{n}) \geq v_p(x_{n})+1,
\end{equation}
for all $n \ge 0$. We are going to prove that, starting from $x\alpha+y$, if Equation~\eqref{Eq: patoldis} holds for all $n\geq0$, then $\{a_n\}_{n\geq0}$ is the continued fraction expansion of $-\frac{y}{x}$, so that $x\alpha+y=0$. Proposition~\ref{Prop: finexpansion} characterizes the $p$-adic numbers having finite continued fractions with respect to the Ruban, \textit{Browkin I}, and Algorithm~\eqref{Alg: MR} expansions. 
Therefore, via Ruban's algorithm, Equation~\eqref{Eq: patoldis} holds for all $n\geq0$ if and only if $-\frac{y}{x}<0$ and $\alpha = -\frac{y}{x}$. In fact, if $-\frac{y}{x}>0$, the continued fraction expansion of $-\frac{y}{x}$ is finite and so we are able to determine the floor function after a finite number of input transformations. With \textit{Browkin I}, Equation~\eqref{Eq: patoldis} cannot hold for all $n\geq0$, because the continued fraction expansion of every rational number is finite.
We start by showing an example in which this happens with Ruban's algorithm.

\begin{Example}
In certain cases, no digit of $\lfloor x\alpha + y \rfloor_p$ can be computed, even after applying an arbitrary number of input transformations, when the Ruban's continued fraction expansion is considered.
In particular, we do not know any digit of $x_n\alpha_n+y_n$, for any $n\geq0$.
Let \[\alpha = \left [\frac{1}{p}, \overline{\frac{p^2-1}{p}} \right ], \quad x=p, \quad \text{and} \quad y=p^2-1.\]   
We have
\[
x \alpha + y = x \left( a_0 + \frac{1}{\alpha_1} \right) + y = p^2 + \frac{p}{\alpha_1},
\]
with \( v_p\left( \frac{p}{\alpha_1} \right) \ge 2 \). At this point, it is not possible to compute any digit of \( x\alpha + y \), since \( p^2 \) is added to a quantity of valuation $\ge 2$ that remains unknown. Inductively, we can show that, after $k \ge 1$ input transformations, we obtain $x_k = p^{k+1}$ and $y_k = p^k$, and
\[p^{k+1} \alpha_k + p^k = p^{k+1} a_k + p^k + \frac{p^{k+1}}{\alpha_{k+1}} = p^{k+2} + \frac{p^{k+1}}{\alpha_{k+1}}.\]
It is not possible to compute any digit of \( x\alpha + y \), since \( p^{k+2} \) is added to $\frac{p^{k+1}}{\alpha_{k+1}}$, which has valuation $\ge k+2$ and, hence, it remains unknown. Since the Ruban $p$-adic continued fraction of $-p$ is 
\[-p = \left [0, \overline{\frac{p^2-1}{p}} \right ],\]
we obtain $\alpha = \frac{1}{p} - p$ and therefore
$x \alpha + y = 0$.
\end{Example}

\begin{Proposition}\label{Prop: uniquedis}
Let $x, y \in \Q$, with $x \ne 0$. Then there exists exactly one $a \in \mathcal{B}$ and $a\in \mathcal{R}$ such that
\begin{equation}\label{Eq: disastroso}
v_p(x a + y) \geq v_p(x)+1.
\end{equation}
\end{Proposition}
\begin{proof}
First of all notice that, if $v_p(y)>v_p(x)$, then
\[v_p(xa+y)=v_p(xa)=v_p(x)+v_p(a),\]
so that \eqref{Eq: disastroso} is satisfied only for $a=0$. If $v_p(y)\leq v_p(x)$, then we necessarily have $v_p(xa)=v_p(y)$, as otherwise 
\[v_p(xa+y)=\min\{v_p(xa),v_p(y)\}\leq v_p(xa)\leq v_p(x).\]
Therefore, the valuation of $a$ is determined and it is 
\[v_p(a) = v_p(y) - v_p(x).\]
For simplicity, let us denote $v_x=v_p(x)$, $v_y=v_p(y)$ and let us write $x = \sum\limits_{n=v_x}^{\infty}x_n p^n$, $y = \sum\limits_{n=v_y}^{\infty}y_n p^n$, $a = \sum\limits_{n=-r}^{0}c_n p^n$ and $x a = \sum\limits_{n=v_y}^{\infty} d_n p^n$. Then
\[\left ( \sum_{n=v_x}^{\infty}x_n p^n \right ) \left ( \sum_{n=-r}^{0}c_n p^n \right ) + \sum_{n=v_y}^{\infty} y_n p^n = \sum_{n=v_y}^{\infty} d_n p^n + \sum_{n=v_y}^{\infty} y_n p^n = \sum_{n=v_x + 1}^{\infty} f_n p^n,\]
for some $f_n \in \Z/p\Z$, with $n \ge v_x + 1$.
This means that the following equalities hold:
\begin{equation}\label{Eq: systemdisastroso}\begin{cases}
d_{v_y} + y_{v_y} = 0  \\
d_{v_y + 1} + y_{v_y + 1} + 1 = 0 \\
d_{v_y + 2} + y_{v_y + 2} + 1 = 0 \\
\quad \vdots \\
d_{v_x} + y_{v_x} + 1  = 0, \\
\end{cases}
\end{equation}
where all the equalities are considered modulo $p$.
Substituting $d_{v_y + n} = \sum\limits_{i=0}^n c_{-r+i}x_{v_x +n-i}$, for $n=0,\ldots, r$ in \eqref{Eq: systemdisastroso}, we obtain
\begin{equation*}
\begin{cases}
c_{-r}x_{v_x} + y_{v_y}= 0 \\
c_{-r+1}x_{v_x} + c_{-r}x_{v_x+1} + y_{v_y + 1} + 1 = 0 \\
c_{-r+2}x_{v_x} + c_{-r+1}x_{v_x + 1} + c_{-r}x_{v_x+2} + y_{v_y + 2} + 1 = 0 \\
\quad \vdots \\

\sum_{i=0}^r c_{-r+i}x_{v_x +r-i} + y_{v_x} + 1 = 0. \\
\end{cases}
\end{equation*}
Hence,
\begin{align*}
c_{-r} &=  -y_{v_y}\cdot x_{v_x}^{-1},\\ 
c_{-r+1} &=  (-c_{-r}x_{v_x+1} - y_{v_y + 1} - 1)\cdot x_{v_x}^{-1},
\end{align*}
and analogously for the other indices.
It is equivalent to solve the linear system in $\mathbb{F}_p$
\[\begin{pmatrix}
x_{v_x} & 0 & 0 & \ldots & 0 \\
x_{v_x +1} & x_{v_x} & 0 & \ldots & 0 \\
x_{v_x +2} & x_{v_x +1} & x_{v_x} & \ldots & 0 \\
\vdots & \vdots & \vdots & \ddots &  \vdots \\
x_{v_x +r} & x_{v_x +r-1} & x_{v_x +r-2} & \ldots & x_{v_x} \\
\end{pmatrix} \begin{pmatrix} c_{-r} \\ c_{-r+1} \\ c_{-r+2} \\ \vdots \\ c_{0} \end{pmatrix} = \begin{pmatrix}  - y_{v_y} \\ -1- y_{v_y+1} \\ -1- y_{v_y+2} \\ \vdots \\ -1- y_{v_x} \end{pmatrix},\]
which admits an unique solution since 
\[\det \begin{pmatrix}
x_{v_x} & 0 & 0 & \ldots & 0 \\
x_{v_x +1} & x_{v_x} & 0 & \ldots & 0 \\
x_{v_x +2} & x_{v_x +1} & x_{v_x} & \ldots & 0 \\
\vdots & \vdots & \vdots & \ddots &  \vdots \\
x_{v_x +r} & x_{v_x +r-1} & x_{v_x +r-2} & \ldots & x_{v_x} \\
\end{pmatrix} = x_{v_x}^{r+1} \ne 0.\]
In fact, since $v_x=v_p(x)$, then the coefficient $x_{v_x}$ is nonzero.
\end{proof}

A direct consequence of Proposition~\ref{Prop: uniquedis} is the following characterization of the cases in which \eqref{Eq: patoldis} holds for all $n$.

\begin{Corollary}\label{Cor: valunique}
Let $x, y \in \Q$, with $x\neq 0$. Then there exists a unique $\alpha \in \Q_p$, with $\alpha = [a_0, a_1, \ldots]$ such that
\begin{equation}\label{Eq: toeli}
v_p(x_na_n + y_n) \geq v_p(x_n)+1
\end{equation}
for all $n\geq 1$, and it is $\alpha=-\frac{y}{x}$.
\begin{proof}
Given $x_n, y_n \in \mathbb{Q}$, Proposition~\ref{Prop: uniquedis} determines the unique partial quotient $a_n$ such that
\[
v_p(x_n a_n + y_n) \geq v_p(x_n) + 1.
\]
This value of $a_n$ uniquely determines $x_n = x_n a_n + y_n$ and $y_{n+1} = x_n$. Therefore, for all $n\geq0$ the partial quotient $a_n$ is uniquely determined by $x_n$ and $y_n$ and it uniquely determines $x_{n+1}$ and $y_{n+1}$. This means that there exists a unique sequence $\{a_n\}_{n\geq0}$ of partial quotients such that \eqref{Eq: toeli} holds for all $n$. We now show that the condition is satisfied for $\alpha=-\frac{y}{x}$, and hence that this is the unique such element. Indeed, since $a_0=\lfloor-\frac{y}{x}\rfloor_p$, we have
\[v_p\left(-\frac{y}{x}-a_0\right)=v_p\left(\frac{y}{x}+a_0\right)=v_p\left(\frac{xa_0+y}{x}\right)\geq 1,\]
and thus $v_p(xa_0+y)\geq v_p(x)+1$. The next complete quotient is
\[\alpha_1=\frac{1}{-\frac{y}{x}-a_0}=\frac{x}{-(xa_0+y)}=-\frac{y_1}{x_1}.\]
Similarly, $a_1=\left\lfloor-\frac{y_1}{x_1}\right\rfloor_p$ satisfies condition~\eqref{Eq: disastroso} for $n=1$. It is not hard to see inductively that the $n$-th complete quotient of $-\frac{y}{x}$ is $\alpha_n=-\frac{y_n}{x_n}$ and the $n$-th partial quotient $a_n=\left\lfloor-\frac{y_n}{x_n}\right\rfloor_p$ satisfies condition~\eqref{Eq: disastroso}. Therefore, the unique $\alpha$ satisfying \eqref{Eq: disastroso} for all $n$ is  $\alpha=-\frac{y}{x}$.
\end{proof}
\end{Corollary}

In the following proposition, we give an alternative direct proof of the latter result.

\begin{Proposition}\label{Prop: tolet}
Let $x,y \in \mathbb{Q}$, with $x \ne 0$. Let $\alpha \in \Q_p$ be given by the infinite Ruban or \textit{Browkin I} continued fraction expansion $\alpha = [a_0,a_1,a_2, \ldots]$, where $v_p(a_0)\leq0$ and $v_p(a_n)<0$ for all $n \ge 1$. If
\begin{equation}\label{Eq: condvx}
v_p(x_{n+1})\geq v_p(x_{n})+1 \quad \text{for all }n \ge 0,
\end{equation}
then $x\alpha+y=0$, that is $\alpha = -\frac{y}{x}$.
\end{Proposition}
\begin{proof}
By construction, we have \[x \alpha + y = \frac{1}{\alpha_1} (x_1 \alpha_1 + y_1 ) =   \frac{1}{\prod_{i=1}^n \alpha_i} (x_n \alpha_n + y_n )  \quad \text{for all }n \ge 0.\]
Then, we obtain
\[v_p(x \alpha + y) = - \sum_{i=1}^n v_p(\alpha_i) + v_p(x_n \alpha_n + y_n) > v_p(x_n \alpha_n + y_n).\]
By hypothesis, for all $n\ge 0$, we have 
\[v_p(x_n \alpha_n + y_n) = v_p(x_n \alpha_n + x_{n-1}) \ge v_p(x_n a_n + x_{n-1})>v_p(x_n).\] Since, by \eqref{Eq: condvx}, 
\[\lim\limits_{n \to +\infty} v_p(x_n) = + \infty,\]
we have $v_p(x \alpha + y) = + \infty$, and this means that $x \alpha + y=0$.
\end{proof}

\begin{Remark}
At this point we know that, unless the numerator or the denominator of the M\"obius transformation vanishes (that is, unless $\alpha$ is either $-\frac{y}{x}$ or $-\frac{t}{z}$), after a finite number of input transformations we reach a situation in which no carry-overs occur. 
If \textit{Browkin I} expansion is considered, then $-\frac{y}{x}$ has a finite continued fraction expansion and so, after a finite number of input transformations we are able to compute the next partial quotient.
On the other hand, if Ruban's expansion is considered, then $-\frac{y}{x}$ may have an infinite continued fraction expansion (see Proposition~\ref{Prop: finexpansion}). 
In that case, the condition
\[v_p(x_na_n+y_n)\geq v_p(x_n)+1,\]
holds as long as the elements $a_n$ coincide with the partial quotients of the $p$-adic continued fraction of $-\frac{y}{x}$, and this may occur for arbitrarily many indices $n$.
\end{Remark}

\subsection{Output transformation}

In this section, we study the output transformation performed after the computation of the first partial quotient of
\[\gamma=\frac{x\alpha+y}{z\alpha+t}\]
when the hypotheses of Theorem~\ref{thmfloor} are satisfied. 
The output transformation is identical to that in Equation~\eqref{Eq: outputreal}, as it involves only arithmetic manipulations. In light of the results of Section~\ref{Sec: MobInp}, the only bad case to avoid is the case where $\alpha$ is a zero of either the denominator or the numerator of the transformation. 
The next proposition shows that this is never the case and, whenever the floor function $l=\lfloor \gamma \rfloor_p$ can be computed, both the numerator and denominator of
\[\gamma_1=\frac{1}{\gamma-l}=\frac{z \alpha + t}{(x - l z)\alpha + y -lt},\]
are nonzero.

\begin{Proposition}\label{Pro: outgood}
Let $x,y,z,t \in \Q$, $\alpha \in \Q_p$, and let $l = \left \lfloor \frac{x \alpha + y}{z \alpha + t} \right \rfloor_p$. After an output transformation, we obtain
\[\frac{z \alpha + t}{(x - l z)\alpha + y -lt},\]
with $z \alpha + t \ne 0$ and $(x - lz)\alpha + y -lt \ne 0$.
\end{Proposition}
\begin{proof}
The numerator $z \alpha + t \ne 0$, since otherwise $l$ would not be determined, by Theorem~\ref{thmfloor}. If $(x - l z)\alpha + y -lt = 0$, then $\alpha = \frac{-y +lt}{x - l z},$ and
\[\begin{aligned}\frac{x \alpha + y}{z \alpha + t} &= \frac{x \frac{-y +lt}{x - l z} + y}{z \frac{-y +lt}{x - l z} + t} = l.\end{aligned}\]
This implies that
\[\left\lfloor\frac{x \alpha + y}{z \alpha + t}\right\rfloor_p=\frac{x \alpha + y}{z \alpha + t}=l.\]
In this case the transformation has finite continued fraction expansion
\[\frac{x \alpha + y}{z \alpha + t} = [l],\]
hence the algorithm terminates without computing the output transformation.
\end{proof}

\subsection{The algorithm for the M\"obius transformation}\label{Sec: themobalgo}
In this section, we summarize the method we developed for computing Ruban's and \textit{Browkin I} $p$-adic continued fraction of the M\"obius transformation, and examine its finiteness of the algorithm.
At each step, Theorem~\ref{thmfloor} establishes necessary and sufficient conditions to determine the $p$-adic floor function $\lfloor\gamma\rfloor_p$ of the M\"obius transformation. In particular, we have seen that after a finite number of input transformations we always end up in the case
\begin{align}\label{Eq: goodconditions}
\begin{split}
v_p(x_n\alpha_n)&=v_p(x_n a_n)<v_p(y_n),\\
v_p(z_n\alpha_n)&=v_p(z_n a_n)<v_p(t_n),
\end{split}
\end{align}
except when either the numerator or the denominator of the transformation is zero. At this point, we are able to correctly compute the partial quotient
\[l=\left\lfloor\frac{x\alpha+y}{z\alpha+t}\right\rfloor_p,\]
if and only if the condition on the valuations of Theorem~\ref{thmfloor} is satisfied. In particular, we know that there exists $\overline{n}$ such that the conditions~\eqref{Eq: goodconditions} are satisfied for all $n\geq \overline{n}$. Then, by Proposition~\ref{Prop: othername}, we can compute $\lfloor\gamma\rfloor_p$ if and only if
\begin{equation}\label{Eq: condestra}
-v_p(a_n)\geq v_p(z_{\overline{n}})-v_p(x_{\overline{n}}),
\end{equation}
for some $n\geq\overline{n}$. Therefore, the strategy is to perform input transformations until \eqref{Eq: goodconditions} is satisfied. Then, we try to compute the floor function, checking if for some $n\geq {\overline{n}}$, condition~\eqref{Eq: condestra} holds. This is the condition of Proposition~\ref{Prop: othername}, and it merely depends on the $p$-adic valuation of the partial quotients of $\alpha$.
This procedure is summarized in Algorithm~\ref{Alg: Gosp1} and some computational experiments for the performance of this algorithm are shown in Section~\ref{Sec: compMob}. The SageMath implementation of Algorithm~\ref{Alg: Gosp1} is publicly available\footnote{\href{https://github.com/giulianoromeont/p-adic-continued-fractions}{https://github.com/giulianoromeont/p-adic-continued-fractions}}.
\bigskip

Since the right-hand side of \eqref{Eq: condestra} is constant, the fulfillment of the condition only depends on $v_p(a_n)$, for $n\geq\overline{n}$. If the valuation $v_p(a_n)$ is bounded, condition~\eqref{Eq: condestra} is not guaranteed to be satisfied for some $n$, and it may not be possible to determine $\lfloor\gamma\rfloor_p$. Such a situation occurs in Example~\ref{Exa: neveroutput}, where all the partial quotients $a_n$, for $n\geq 1$ have valuation $-1$. If the valuation of the partial quotients of $\alpha$ is unbounded, we are guaranteed that Algorithm~\ref{Alg: Gosp1} terminates correctly, computing the $p$-adic continued fraction of $\gamma$. In Section~\ref{Sec: measure}, and in particular in Corollary~\ref{Coro: unbounded}, we showed that the $p$-adic valuation of the partial quotients for Ruban's and \textit{Browkin I} continued fractions is unbounded for all $\alpha \in \mathbb{Q}_p$, except for a set of Haar measure zero. This implies that, given the coefficients $x,y,z,t$ of the M\"obius transformation, our algorithm terminates correctly for $\mu$-\textit{almost all} inputs $\alpha \in \Q_p$.
Notice that Algorithm~\ref{Alg: Gosp1} computes the partial quotients of the transformation if and only if it is possible to recover them from the knowledge of the partial quotients of $\alpha$. In fact, whenever the algorithm is not able to compute the floor function of the transformation, there is some missing information on its $p$-adic digits, even by using an arbitrary number of partial quotients.

\section{Bilinear fractional transformation of $p$-adic continued fractions}\label{Sec: bilinearfractional}
In this section, we study how to compute Ruban's and \textit{Browkin I} $p$-adic continued fraction expansion of the bilinear fractional transformation of two $p$-adic numbers $\alpha,\beta\in\Q_p$, given by
\[\gamma=\frac{x\alpha\beta+y\alpha+z\beta+t}{e\alpha\beta+f\alpha+g\beta+h},\]
with $x,y,z,t,e,f,g,h\in\Q$.
In the following, we often rewrite the bilinear transformation as 
\begin{equation}\label{Eq: bilform}
x\alpha\beta+y\alpha+z\beta+t=\alpha(x\beta+y)+(z\beta+t),
\end{equation}
in order to split it in two addends.

\begin{Proposition}\label{Prop: bilnum}
Let $\alpha=\sum\limits_{n=-r_\alpha}^{+\infty}c_np^n$ and $\beta=\sum\limits_{n=-r_\beta}^{+\infty}d_np^n\in\Q_p$ and $x,y,z,t \in \Q$. Let us consider
\begin{equation}\label{Eq: bilinearnumerator}
\gamma = x \alpha \beta + y \alpha + z \beta + t.
\end{equation}
Let us suppose to know the first $h_\alpha$ digits of $\alpha$ and the first $h_\beta$ digits of $\beta$. 
Let
\[v = v_p \left ( x\left( \sum\limits_{n=-r_\alpha}^{-r_\alpha+h_\alpha-1}c_np^n\right)\left( \sum\limits_{n=-r_\beta}^{-r_\beta+h_\beta-1}d_np^n\right) + y \sum\limits_{n=-r_\alpha}^{-r_\alpha+h_\alpha-1}c_np^n + z \sum\limits_{n=-r_\beta}^{-r_\beta+h_\beta-1}d_np^n+ t\right)\]
and
\[M= \min \{ v_p(x) -r_{\alpha} - r_{\beta} + \min \{ h_{\alpha}, h_{\beta}\}, v_p(y) - r_{\alpha} + h_{\alpha}, v_p(z) - r_{\beta} + h_{\beta} \}.\]
Then, we know $k$ digits of $\gamma$, where
\[k = M - v\]
and, if $k \leq 0$, no digit of $\gamma$ can be determined.
\end{Proposition}
\begin{proof}
The known digits of $\alpha$ and $\beta$ are, respectively,
\[\sum\limits_{n=-r_\alpha}^{-r_\alpha+h_\alpha-1}c_np^n \text{ and }  \sum\limits_{n=-r_\beta}^{-r_\beta+h_\beta-1}d_np^n.\]
By Lemma~\ref{Lem: prodab}, we know the digits of valuation 
\[v_p(x) + v_p(\alpha) + v_p(\beta), \ldots, v_p(x) + v_p(\alpha) + v_p(\beta) + \min \{ h_{\alpha}, h_{\beta}\}-1\]
of $v_p(x \alpha \beta)$, but not the digit of valuation $v_p(\alpha) + v_p(\beta) + \min \{ h_{\alpha}, h_{\beta}\}$.
Similarly, we know the digits of valuation
\[v_p(y) + v_p(\alpha), \ldots, v_p(y) + v_p(\alpha) + h_{\alpha}- 1\]
of $v_p(y \alpha)$, but not the digit of valuation $v_p(y) + v_p(\alpha) + h_{\alpha}$ and the digits of valuation \[v_p(z)+v_p(\beta), \ldots, v_p(z) + v_p(\beta)+h_\beta-1\]
of $v_p(z \beta)$ but not the digit of valuation $v_p(z) + v_p(\beta) + h_\beta$ of $v_p(z \beta)$. Therefore, we are able to determine any digit from $v$ to $M-1$, that is a total of $k=M-v$ digits.
\end{proof}

As a corollary, we identify the conditions under which no $p$-adic digit of the bilinear transformation can be computed.

\begin{Corollary}
Let $x,y,z,t \in \Q$ and $\alpha, \beta \in \Q_p$ and let
\begin{equation*} \gamma = x \alpha \beta + y \alpha + z \beta + t.\end{equation*}
Let $-r_\alpha \le 0$ and $-r_\beta \le 0$ the valuation of $\alpha$ and $\beta$ respectively.
We know $\lfloor \alpha \rfloor_p$ and $\lfloor \beta \rfloor_p$. We know no digits of $\gamma$ if and only if \[v_p(\gamma) > \min \{ v_p(x) - \max \{ r_\alpha, r_\beta \}, v_p(y) , v_p(z) \}.\]
\begin{proof}
The result follows immediately by substituting $h_\alpha = r_\alpha + 1$ and $h_\beta = r_\beta + 1$ into the expressions for $v$ and $M$ in Proposition~\ref{Prop: bilnum}.
\end{proof}
\end{Corollary}

\begin{Remark}
Without loss of generality, we may assume that $v_p(\alpha) \le 0$ and $v_p(\beta) \le 0$, so that $v_p(\alpha) = v_p(\lfloor \alpha \rfloor_p)$ and $v_p(\beta) = v_p(\lfloor \beta \rfloor_p)$. Indeed, if the continued fraction expansion of $\alpha$ and $\beta$ are given by 
\[\alpha = [a_0,a_1,a_2, \ldots]\]
and
\[\beta = [b_0,b_1,b_2, \ldots],\]
then $v_p(a_i)<0$ and $v_p(b_i)<0$ for all $i \ge 1$. Therefore, after applying an input transformation to both $\alpha$ and $\beta$, we can reduce to the case where $v_p(\alpha)<0$ and $v_p(\beta)<0$.
\end{Remark}

The following theorem provides the necessary and sufficient conditions for the determination of the floor function of the bilinear fractional transformation of $\alpha$ and $\beta$, given its rational coefficients and $\lfloor \alpha \rfloor_p$ and $\lfloor \beta \rfloor_p$.

\begin{Theorem}\label{thmfloorBilinear}
Let $\alpha, \beta \in \Q_p$ and $x,y,z,t,e,f,g,h \in \Q$, with
\[\mathrm{rank}\begin{pmatrix}
    x & y & z & t\\
    e & f & g & h
\end{pmatrix}=2,\]
 and such that $v_p(\alpha)=-r_\alpha\leq 0$, $v_p(\beta)=-r_\beta\leq 0$. Let us suppose that $\lfloor\alpha\rfloor_p=\sum\limits_{n=-r_\alpha}^{0}c_np^n$ and $\lfloor\beta\rfloor_p=\sum\limits_{n=-r_\beta}^{0}d_np^n$ are known. Let us denote
\begin{equation}\label{Eq: bihomo}
\gamma = \frac{x \alpha \beta + y \alpha + z \beta + t}{e \alpha \beta + f \alpha + g \beta + h}.
\end{equation}
Let
\begin{align*}
v^{(1)}&=v_p(x\lfloor\alpha\rfloor_p\lfloor\beta\rfloor_p+y\lfloor\alpha\rfloor_p+z\lfloor\beta\rfloor_p+t),\\
v^{(2)}&=v_p(e\lfloor\alpha\rfloor_p\lfloor\beta\rfloor_p+f\lfloor\alpha\rfloor_p+g\lfloor\beta\rfloor_p+h),
\end{align*}
and
\begin{align*}
M_{num} &= \min \{ v_p(x) - \max \{ r_\alpha, r_\beta \}, v_p(y) , v_p(z) \},\\
M_{den} &=  \min \{ v_p(e) - \max \{ r_\alpha, r_\beta \}, v_p(f) , v_p(g) \}.
\end{align*}
If either
\[k_{num} = v^{(1)} - M_{num} \le 0 \quad \text{or} \quad k_{den}=v^{(2)} - M_{den} \le 0, \]
then it is not possible to determine any digit of $\gamma$ and of its floor function $\lfloor\gamma\rfloor_p$.
Otherwise, the $p$-adic floor function $\lfloor\gamma\rfloor_p$ is uniquely determined by $\lfloor\alpha\rfloor_p$ and $\lfloor\beta\rfloor_p$ if and only if 
\begin{equation}\label{Eq: outcondbili}
\min \{ k_{num}, k_{den}\} \ge v^{(2)}-v^{(1)}+1.
\end{equation}

In particular, in this case,
\begin{equation*}
\lfloor\gamma\rfloor_p=\left\lfloor\frac{x \alpha \beta + y \alpha + z \beta + t}{e \alpha \beta + f \alpha + g \beta + h}\right\rfloor_p=\left\lfloor\frac{x\lfloor\alpha\rfloor_p\lfloor\beta\rfloor_p+y\lfloor\alpha\rfloor_p+z\lfloor\beta\rfloor_p+t}{e\lfloor\alpha\rfloor_p\lfloor\beta\rfloor_p+f\lfloor\alpha\rfloor_p+g\lfloor\beta\rfloor_p+h}\right\rfloor_p.
\end{equation*}
\end{Theorem}

\begin{proof}
The proof of this result uses a similar argument to that of Theorem~\ref{thmfloor} for the M\"obius transformation. We use Proposition~\ref{Prop: bilnum} to understand how many $p$-adic digits of $\gamma$ can be determined from the knowledge of $\lfloor\alpha\rfloor_p$ and $\lfloor \beta \rfloor_p$, and then we check whether this number is greater than the valuation of $\lfloor\gamma\rfloor_p$. If this is the case, then the partial quotient $\lfloor\gamma\rfloor_p$ is computed correctly. In order to compute the number of digits that we are able to determine for the numerator and the denominator of \eqref{Eq: bihomo}, we apply Proposition~\ref{Prop: bilnum}. In this case, $h_\alpha=r_\alpha+1$ and $h_\beta=r_\beta+1$. 
Therefore, we are able to determine
\begin{align*}
k_{num} &=v^{(1)} - M_{num},\\
k_{den} &=v^{(2)} - M_{den},
\end{align*}
$p$-adic digits of, respectively, the numerator and the denominator of $\gamma$. If $k_{num} \le 0$ or $k_{den} \le 0$ we cannot determine any digit of $\gamma$. Otherwise, by the results of Section~\ref{Sec: aux}, we can compute
\[k=\min\{k_{num},k_{den}\}\]
digits of $\eqref{Eq: bihomo}$. Since the $p$-adic valuation of $\gamma$ is $v^{(1)}-v^{(2)}$, we know the digits up to index $v^{(1)}-v^{(2)} + k-1$. Therefore, are able to determine $\lfloor \gamma \rfloor_p$ if and only if
\[ v^{(1)}-v^{(2)} + k-1 \ge 0,\]
and the claim is proved.
\end{proof}

\begin{Remark}\label{rem: outputcond}
The output condition~\eqref{Eq: outcondbili} of Theorem~\ref{thmfloorBilinear} heavily simplifies whenever we are in the good case where
\begin{align*}
v_p(x_n\beta_{n_\beta})<v_p(y_n), & \ & 
v_p(z_n\beta_{n_\beta})<v_p(t_n), & \ & 
v_p(x_n\alpha_{n_\alpha})<v_p(z_n),\\
v_p(e_n\beta_{n_\beta})<v_p(f_n), & \ &
v_p(g_n\beta_{n_\beta})<v_p(h_n), & \ &
v_p(e_n\alpha_{n_\alpha})<v_p(g_n), 
\end{align*}
for some $n\in\N$ and $n_\alpha+n_\beta=n$. This is the situation after $n_\alpha$ $\alpha$-input transformations and $n_\beta$ $\beta$-input transformations. In fact, in this case,
\begin{align*}
v_p\left(x_n \lfloor \alpha_{n_\alpha} \rfloor_p \lfloor \beta_{n_\beta} \rfloor_p  + y_n \lfloor \alpha_{n_\alpha} \rfloor_p + z_n \lfloor \beta_{n_\beta} \rfloor_p + t_n\right)&=v_p(x_n\lfloor \alpha_{n_\alpha} \rfloor_p \lfloor \beta_{n_\beta} \rfloor_p),\\
v_p\left(e_n \lfloor \alpha_{n_\alpha} \rfloor_p \lfloor \beta_{n_\beta} \rfloor_p  + f_n \lfloor \alpha_{n_\alpha} \rfloor_p + g_n \lfloor \beta_{n_\beta} \rfloor_p + h_n\right)&=v_p(e_n\lfloor \alpha_{n_\alpha} \rfloor_p \lfloor \beta_{n_\beta} \rfloor_p).
\end{align*}
We do not require any hypothesis on the valuation $v_p(y_n\alpha_{n_\alpha}+t_n)$ and $v_p(f_n\alpha_{n_\alpha}+h_n)$, as they do not give any contribution for the total valuation.
Moreover, as proved in Proposition~\ref{Prop: transitions}, once these conditions hold for some $\overline{n}$, they hold for any $n\geq\overline{n}$, either performing input $\alpha$-input or $\beta$-input transformations. In this case, $v_n^{(2)}-v_n^{(1)}$ becomes constant and it is, for all $n\geq\overline{n}$ \[v_n^{(2)}-v_n^{(1)}=v_p(e_{\overline{n}})-v_p(x_{\overline{n}}).\]
Moreover, using the notation of Theorem~\ref{thmfloorBilinear},
\[k_{num}=k_{den}=\min\{v_p(\alpha_{n_\alpha}),v_p(\beta_{n_\beta}) \}-v_p(\alpha_{n_\alpha})-v_p(\beta_{n_\beta})+1.\]
This means that, for $n\geq \overline{n}$, we are able to perform the output if and only if
\begin{equation}\label{Eq: conrarb}
\min\{-v_p(\alpha_{n_\alpha}),-v_p(\beta_{n_\beta})\}\geq v_p(e_{\overline{n}})-v_p(x_{\overline{n}}).
\end{equation}
This implies that, once the target valuation $v_p(e_{\overline{n}})-v_p(x_{\overline{n}})$ is fixed, we are able to perform output provided that both the valuations of the partial quotients of $\alpha$ and $\beta$ are ``sufficiently negative".
\end{Remark}

\subsection{Input transformation}\label{Sec: inptrans}
As we have seen in Section~\ref{Sec: bilinear}, in the case of the bilinear fractional transformation we can exploit both the partial quotients of $\alpha$ and $\beta$, therefore we have two different input transformations. In this section, we are going to examine how the input transformations can be exploited to eventually satisfy the hypotheses of Theorem~\ref{thmfloorBilinear}, hence allowing to compute
\[\lfloor\gamma\rfloor_p=\left\lfloor\frac{x\alpha\beta+y\alpha+z\beta+t}{e\alpha\beta+f\alpha+g\beta+h}\right\rfloor_p.\]
In this case as well, unlike the real case discussed in Section~\ref{Sec: bilinear}, it is not guaranteed that the hypotheses of Theorem~\ref{thmfloorBilinear} are eventually satisfied after a finite number of input transformations. The input transformations are the same as those in Section~\ref{Sec: bilinear}, but we rewrite them in the form of \eqref{Eq: bilform}. Given $\alpha, \beta \in \Q_p$, with $\alpha = [a_0,a_1,a_2, \ldots]$ and $\beta = [b_0,b_1,b_2, \ldots]$, we start from
\begin{equation}\label{Eq: bilgamma}
\gamma=\frac{x\alpha\beta+y\alpha+z\beta+t}{e\alpha\beta+f\alpha+g\beta+h}=\frac{\alpha(x\beta+y)+z\beta+t}{\alpha(e\beta+f)+g\beta+h}.
\end{equation}
Performing one $\alpha$-input transformation, i.e. substituting $\alpha=a_0+\frac{1}{\alpha_1}$ in \eqref{Eq: bilgamma}, we obtain
\begin{equation}\label{Eq: inputalpha}
\gamma=\frac{\alpha_1((xa_0+z)\beta+(ya_0+t))+x\beta+y}{\alpha_1((ea_0+g)\beta+(fa_0+h))+e\beta+f}.
\end{equation}
Performing one $\beta$-input transformation, i.e. substituting $\beta=b_0+\frac{1}{\beta_1}$ in \eqref{Eq: bilgamma}, we obtain
\begin{equation}\label{Eq: inputbeta}
\gamma=\frac{\alpha((xa_0+y)\beta_1+x)+(za_0+t)\beta_1+z}{\alpha((ea_0+f)\beta_1+e)+(ga_0+h)\beta_1+g}.
\end{equation}
The first thing that it is possible to notice is that, after performing $n$ consecutive $\beta$-input transformations, we get 
\begin{equation}\label{Eq: inputbetaN}
\gamma=\frac{\alpha(x_n\beta_n+y_n)+z_n\beta_n+t_n}{\alpha(e_n\beta_n+f_n)+g_n\beta_n+h_n},
\end{equation}
where we are using the notation of Definition~\ref{Def: recur}. It follows by Corollary~\ref{Cor: valunique} and Proposition~\ref{Prop: tolet} that, if the quantities
\[x\beta+y, \quad  z\beta+t, \quad e\beta+f, \quad  g\beta+h\]
are all different from zero, then there exists $\overline{n}\geq0$ such that, for all $n\geq \overline{n}$, \eqref{Eq: inputbetaN} holds with
\begin{align*}
v_p(x_nb_n)&<v_p(y_n),\\
v_p(z_nb_n)&<v_p(t_n),\\
v_p(e_nb_n)&<v_p(f_n),\\
v_p(g_nb_n)&<v_p(h_n).
\end{align*}
Let us assume that we start from this situation. In the remaining cases, that are,
\[\beta\in\left\{-\frac{y}{x},-\frac{t}{z},-\frac{f}{e},-\frac{h}{g}\right\},\]
the bilinear transformation simplifies and it is easier to study, as either the numerator or the denominator reduce to a linear transformation (possibly multiplied by $\alpha$) as studied in Section~\ref{Sec: Mobpadic}. For the same reason, we also assume $\alpha\not\in\{-\frac{z}{x},-\frac{t}{y},-\frac{g}{e},-\frac{h}{f}\}$. For simplicity, from now on we work only with the numerator of the bilinear fractional transformation, but the same arguments hold for the denominator. We have seen in Remark~\ref{rem: outputcond} that the optimal situation is whenever
\begin{equation}\label{Eq: optimalineq}
v_p(x_nb_n)<v_p(y_n), \quad
v_p(z_nb_n)<v_p(t_n), \quad
v_p(x_na_n)<v_p(z_n),
\end{equation}
for some $n\in\N$. In fact, in this case, the condition to compute the floor function of the transformation significantly simplifies in \eqref{Eq: conrarb}. Notice also that, since in this case we have
\[v_p(x_n\alpha\beta+y_n\alpha+z_n\beta+t_n)=v_p(x_n\alpha\beta),\]
we do not care about the behavior of $v_p(y_na_n+t_n)$.
First, we show that whenever no carry-over occurs in $x_n\alpha+z_n$, one $\alpha$-input transformation always yields the desired situation \eqref{Eq: optimalineq}.

\begin{Proposition}\label{Prop: transitions}
Let us suppose that:
\begin{itemize}
    \item $v_p(x_n\beta)<v_p(y_n)$, 
    \item $v_p(z_n\beta)<v_p(t_n)$,
    \item $v_p(x_n\alpha+z_n)=\min\{v_p(x_n\alpha),v_p(z_n)\}$.
\end{itemize}
Then, after an $\alpha$-input transformation, we have:
\begin{enumerate}
    \item[i)] $v_p(x_{n+1}\beta)<v_p(y_{n+1})$,  
    \item[ii)] $v_p(z_{n+1}\beta)<v_p(t_{n+1})$,
   \item[iii)] $v_p(x_{n+1}\alpha)<v_p(z_{n+1})$.    
\end{enumerate}
\begin{proof}
By Lemma~\ref{Lem: rec}, condition $iii)$ is satisfied. Condition $ii)$ holds, since
\[v_p(z_{n+1}\beta)=v_p(x_{n}\beta)<v_p(y_{n})=v_p(t_{n+1}).\]
It remains to show that $i)$ is satisfied, which means showing that
\begin{equation*}
v_p(x_{n+1}\beta)=v_p((x_na+z_n)\beta)=\min\{v_p(x_na_n\beta),v_p(z_n\beta)\},
\end{equation*}
is less than
\begin{equation*}
v_p(y_{n+1})=v_p(y_na+t_n)\geq\min\{v_p(y_na),v_p(t_n)\},
\end{equation*}
where $a=\lfloor\alpha\rfloor_p$. If $v_p(x_na+z_n)=v_p(x_na)$, then
\[v_p((x_na+z_n)\beta)=v_p(x_na\beta)\leq v_p(z_n\beta)<v_p(t_n),\]
and
\[v_p((x_na+z_n)\beta)=v_p(x_na\beta)<v_p(y_n a).\]
Therefore,
\[v_p((x_na+z_n)\beta)<v_p(y_na+t_n),\]
and the claim holds in this case. Instead, if $v_p(x_na+z_n)=v_p(z_n)$, then
\[v_p((x_na+z_n)\beta)=v_p(z_n\beta)<v_p(t_n),\]
and\[v_p((x_na+z_n)\beta)=v_p(z_n\beta)\leq v_p(x_na\beta)<v_p(y_na),\]
which proves condition $i)$.
\end{proof}
\end{Proposition}

However, in the next example we show that if carry-overs occur in $x\alpha+z$, then performing an $\alpha$-input transformation may no longer satisfy the valuation conditions for $x_1\beta$ and $y_1$.

\begin{Example}\label{Rem: namefin}
If in Proposition~\ref{Prop: transitions},
\[v_p(x_n\alpha+z_n)>\min\{v_p(x_n\alpha),v_p(z_n)\},\]
then the claim does not necessarily holds. Let $p=5$ and consider

\[(x,y,z,t)=\left(\frac{1}{5},\frac{1}{25},\frac{24}{625},\frac{1}{5}\right), \quad a= \frac{1}{125}, \quad b=\frac{1}{125}.\]
Then, $v_5(\beta)=-2$ and 
\[v_5(x\beta)<v_5(y), \quad v_5(z\beta)<v_5(t),\]
but
\[v_5(x_1\beta)=v_5(x\alpha+z)+v_5(\beta)>v_5(ya+t)=v_5(y_1).\]
\end{Example}

\begin{Remark}
Under the hypotheses of Proposition~\ref{Prop: transitions}, it is guaranteed that no carry-over occurs in the whole transformation. In fact, in this case,
\[v_p(\alpha(x\beta+y)+(z\beta+t))=v_p(\beta (x\alpha+z))=v_p(\beta)+\min\{v_p(x\alpha),v_p(z)\}.\]
Therefore, a problem arises only in the case when carry-overs occur in $x\alpha+z$. 
In general, if $v_p(x\beta + y) = v_p(x\beta)$ and $v_p(z\beta + t) = v_p(z\beta)$, the total valuation of the bilinear transformation depends solely on $v_p(x\alpha + z)$. 
\end{Remark}

The conditions on $x\beta+y$ and $z\beta+t$ in Proposition~\ref{Prop: transitions} cannot be weakened, as shown in the next remark.

\begin{Remark}
In general, if
\begin{enumerate}
    \item[i)] $v_p(x_n\beta+y_n)=\min\{v_p(x_n\beta),v_p(y_n)\}$,
    \item[ii)] $v_p(z_n\beta+t_n)=\min\{v_p(z_n\beta),v_p(t_n)\}$,
    \item[iii)] $v_p(x_n\alpha+z_n)=\min\{v_p(x_n\alpha),v_p(z_n)\}$,
    \item[iv)] $v_p(y_n\alpha+t_n)=\min\{v_p(y_n\alpha),v_p(t_n)\}$.
\end{enumerate}
then we are not guaranteed that after an $\alpha$-input transformation the same holds. In fact, $iii)$ and $iv)$ hold by Lemma~\ref{Lem: rec}, and $ii)$ holds since
\[v_p(z_{n+1}\beta+t_{n+1})=v_p(x_n\beta+y_n)=\min\{v_p(x_n\beta),v_p(y_n) \}=\min\{v_p(z_{n+1}\beta),v_p(t_{n+1}) \}.\]
However, $i)$ can be not satisfied. Let $p=5$, $\beta=\frac{1}{5}+\cdots$, and consider
\[(x_n,y_n,z_n,t_n,a)=\left(\frac{1}{5},\frac{1}{5},\frac{22}{25},\frac{102}{125},\frac{1}{5}\right).\]
It is not hard to verify that all the four conditions $i)-iv)$ are satisfied at the $n$-th step, but
\begin{align*}
v_5(x_{n+1}\beta+y_{n+1})=-2>\min\{v_5(x_{n+1}\beta),v_5(y_{n+1})\}=-3,
\end{align*}
so that condition $i)$ does not hold at the next step.
\end{Remark}

At this point, we have a good candidate for a strategy to compute the $p$-adic continued fraction of the bilinear fractional transformation:
\begin{itemize}
    \item Perform $\beta$-input transformation until \begin{equation}\label{Eq: hpgoodcase}v_p(x_nb_n)<v_p(y_n), \ 
v_p(z_nb_n)<v_p(t_n), \  v_p(e_nb_n)<v_p(f_n), \  
v_p(g_nb_n)<v_p(h_n),\end{equation}

\item If no carry-over occurs in both $x_n\alpha+z_n$ and $e_n\alpha+g_n$, that is 
\begin{align*}
    v_p(x_n\alpha+z_n)&=\min\{v_p(x_n\alpha),v_p(z_n)\},\\
    v_p(e_n\alpha+g_n)&=\min\{v_p(e_n\alpha),v_p(g_n)\},
\end{align*}
then perform one $\alpha$-input transformation in order to have
\begin{equation}\label{Eq: Good1}\begin{aligned}
&v_p(x_{\overline{n}}\beta)<v_p(y_{\overline{n}}),  & 
&v_p(z_{\overline{n}}\beta)<v_p(t_{\overline{n}}),   &
&v_p(x_{\overline{n}}\alpha)<v_p(z_{\overline{n}}),\\
&v_p(e_{\overline{n}}\beta)<v_p(f_{\overline{n}}),   &
&v_p(g_{\overline{n}}\beta)<v_p(h_{\overline{n}}),   &
&v_p(e_{\overline{n}}\alpha)<v_p(g_{\overline{n}}).
\end{aligned}\end{equation}

\item Compute
\[u=v_p(e_{\overline{n}})-v_p(x_{\overline{n}}),\]
\item Perform $\alpha$-input transformations until $-v_p(\alpha_n)\geq u$.
\item Perform $\beta$-input transformations until $-v_p(\beta_n)\geq u$.
\item Now, condition~\eqref{Eq: conrarb} is satisfied and it is possible to compute \begin{equation*}
\lfloor\gamma\rfloor_p=\left\lfloor\frac{x \alpha \beta + y \alpha + z \beta + t}{e \alpha \beta + f \alpha + g \beta + h}\right\rfloor_p=\left\lfloor\frac{x\lfloor\alpha\rfloor_p\lfloor\beta\rfloor_p+y\lfloor\alpha\rfloor_p+z\lfloor\beta\rfloor_p+t}{e\lfloor\alpha\rfloor_p\lfloor\beta\rfloor_p+f\lfloor\alpha\rfloor_p+g\lfloor\beta\rfloor_p+h}\right\rfloor_p.
\end{equation*}

\end{itemize}

The only problem at this point arises when the following holds:
\begin{align}\label{Eq: ending}
v_p(x_n\alpha+z_n)&>\min\{v_p(x_n\alpha),v_p(z_n)\}.
\end{align}

\begin{Remark}
    
Notice that if \eqref{Eq: ending} holds, then we must perform an $\alpha$-input transformation and we cannot fix it by performing $\beta$-input transformation. In fact, after a $\beta$-input transformation,
\begin{align*}
v_p(x_{n+1}\alpha+z_{n+1})&=v_p((x_nb+y_n)\alpha+(z_nb+t_n))=v_p(b(x_n\alpha+z_n)+(y_n\alpha+t_n))\\
&>\min\{v_p(bx_n\alpha),v_p(y_n\alpha)\}=\min\{v_p(x_{n+1}\alpha),v_p(z_{n+1})\},
\end{align*}
therefore carry-overs still occur. In fact, whenever we know that 
\begin{align*}
v_p(x\beta+y)&=v_p(x\beta),\\
v_p(z\beta+t)&=v_p(z\beta),
\end{align*}
the valuation of the transformation depends only on $v_p(x\alpha+z)$.
\end{Remark}

As we have seen in Example~\ref{Rem: namefin}, whenever
\[v_p(x\alpha+z)>\min\{v_p(x\alpha),v_p(z)\},\]
it is possible that after an $\alpha$-input transformation,
carry-overs occur in $x\beta+y$. At this point, it would be possible to perform $\alpha$-input transformations until
\[v_p(x_n\alpha)<v_p(z_n), \quad v_p(y_n\alpha)<v_p(t_n),\]
and check whether
\begin{align*}
v_p(x_n\beta+y_n)&=\min\{v_p(x_n\beta),v_p(y_n)\}\\
v_p(z_n\beta+t_n)&=\min\{v_p(z_n\alpha),v_p(t_n)\}.
\end{align*}
If this holds, then after one $\beta$-input transformation we are in the good case. Otherwise, we exchange again the role of $\alpha$ and $\beta$. See Algorithm~\ref{Alg: Gosp2} for a pseudocode implementation of this procedure. It computationally appears (see Section~\ref{Sec: Computations}) that by performing this swap of the role of $\alpha$ and $\beta$ we always eventually end up in the good case where conditions~\eqref{Eq: Good1} are satisfied, but we were not able to prove it.

We end this section providing an example of output transformation for the bilinear fractional transformation.
\begin{Example}
If we are in the hypothesis of Equation~\eqref{Eq: hpgoodcase} and we are able to perform the output, then it is not guaranteed that these hypotheses will be satisfied after the output transformation. For example, let $p=5$ and $\alpha, \beta \in \Q_5$ with $\lfloor \alpha \rfloor_5=\lfloor \beta \rfloor_5= \frac{1}{5}$, and let us consider the following transformation
\[ \gamma = \frac{\left ( \frac{1}{5} +5 \right ) \alpha \beta + 25 \alpha + 25\beta + 25}{\frac{1}{5} \alpha \beta + \alpha + \beta + 3}.\]
Using the above notation, we have $x = \frac{1}{5} + 5$, $y = 25$, $z=25$, $t=25$, $e = \frac{1}{5}$, $f = 1$, $g=1$, and $h=3$. It can be verified that
\begin{align*}
&v_5(x\beta)<v_5(y),  & 
&v_5(z\beta)<v_5(t),   &
&v_5(x\alpha)<v_5(z),\\
&v_5(e\beta)<v_5(f),   &
&v_5(g\beta)<v_5(h),   &
&v_5(e\alpha)<v_5(g),
\end{align*}
and
$\min\{r_{\alpha},r_{\beta}\}=\min\{-v_5(\alpha),-v_5(\beta)\}\geq v_5(e)-v_5(x)$. Hence, we are able to extract the first partial quotient of $\gamma$, which is $l_0 = 1$.
After the output transformation, we have
\[ \gamma_1 = \frac{\frac{1}{5} \alpha \beta + \alpha + \beta + 3}{5 \alpha \beta + 24\alpha + 24 \beta +22} = \frac{\Bar{x} \alpha \beta + \Bar{y} \alpha + \Bar{z} \beta + \Bar{t}}{\Bar{e} \alpha \beta + \Bar{f} \alpha + \Bar{g} \beta + \Bar{h}}.\]
Since, $v_5(\Bar{e} \beta) =0= v_5(\Bar{f})$ and $v_5(\Bar{e} \alpha) =0= v_5(\Bar{g})$, conditions \eqref{Eq: hpgoodcase} are no more satisfied.
\end{Example}

\subsection{The algorithm for the bilinear fractional transformation}\label{Sec: thebilalgorithm}
In this section, we outline the pseudocode for the computation of the $p$-adic continued fraction of the bilinear fractional transformation. The strategy is exactly the one described at the end of Section~\ref{Sec: inptrans}. Algorithm~\ref{Alg: Gosp2} summarizes the procedure, and Section~\ref{Sec: compbil} presents the computational experiments. The SageMath implementation is publicly available\footnote{\href{https://github.com/giulianoromeont/p-adic-continued-fractions}{https://github.com/giulianoromeont/p-adic-continued-fractions}}.\bigskip

Similarly to the M\"obius transformation, for the bilinear fractional transformation the algorithm is guaranteed to provide the correct $p$-adic continued fraction for \textit{almost all} $p$-adic numbers. In fact, if, for some $\overline{n}\in\N$,
\begin{align*}
&v_p(x_{\overline{n}}\beta)<v_p(y_{\overline{n}}),  & 
&v_p(z_{\overline{n}}\beta)<v_p(t_{\overline{n}}),   &
&v_p(x_{\overline{n}}\alpha)<v_p(z_{\overline{n}}),\\
&v_p(e_{\overline{n}}\beta)<v_p(f_{\overline{n}}),   &
&v_p(g_{\overline{n}}\beta)<v_p(h_{\overline{n}}),   &
&v_p(e_{\overline{n}}\alpha)<v_p(g_{\overline{n}}),
\end{align*}
then the output condition becomes
\begin{equation}\label{Eq: bilcondab}
\min\{-v_p(a_n),-v_p(b_n)\}\geq v_p(e_{\overline{n}})-v_p(x_{\overline{n}}).
\end{equation}
Therefore, as we have seen in Section~\ref{Sec: themobalgo}, also in this case the strategy is to perform input transformations until $-v_p(a_n)\geq v_p(e_{\overline{n}})-v_p(x_{\overline{n}})$ and $-v_p(b_n)\geq v_p(e_{\overline{n}})-v_p(x_{\overline{n}})$. The fulfillment of this conditions only depends on the $p$-adic valuation of the partial quotients of $\alpha$ and $\beta$. If either the valuation $v_p(a_n)$ or $v_p(b_n)$ is bounded, condition~\eqref{Eq: bilcondab} is not guaranteed to be satisfied for some $n\ge \overline{n}$, and it may not be possible to determine the $p$-adic floor function of the transformation. However, by Corollary~\ref{Coro: unbounded}, we guaranteed that the $p$-adic valuation of the partial quotients for Ruban and \textit{Browkin I} is unbounded for almost all $\alpha,\beta\in\Q_p$, with respect to the Haar measure. Therefore, Algorithm~\ref{Alg: Gosp2} terminates providing the correct $p$-adic continued fraction expansion for almost all inputs $\alpha,\beta \in \Q_p$. Notice also that, whenever \eqref{Eq: bilcondab} is not satisfied, the floor function of the bilinear fractional transformation cannot be recovered, in any way, by the knowledge of arbitrarily many partial quotients of $\alpha$ and $\beta$.

\section{Extending the results for Algorithm~\eqref{Alg: MR}}
\label{Sec: otheralgo}
In this section, we extend the analysis developed in the previous sections for Ruban's and \textit{Browkin I} algorithms to Algorithm~\eqref{Alg: MR}, whose detailed description can be found in \cite{murru2024new}.
If $\alpha = [a_0,a_1, \ldots]$ is the expansion obtained by this algorithm, then the following inequalities hold:
\begin{equation}\label{Eq: mrvaluation}
v_p(a_{2n}) \le 0 \quad \text{and} \quad v_p(a_{2n+1}) < 0 \quad \forall \, n \ge 0.\end{equation}
Using Definition~\ref{Def: represet}, the partial quotients $a_{2n}\in\mathcal{B}$ and $a_{2n+1}\in\mathcal{T}$ for all $n\geq 0$. By construction, we know the first $r_{\alpha_n}+1$ digits of $\alpha_n$ if $n \equiv 0 \pmod{2}$ and the first $r_{\alpha_n}$ digits of $\alpha_n$ if $n \equiv 1 \pmod{2}$, where $-r_{\alpha} = v_p(\alpha_n)$.
By Proposition~\ref{Pro: digitsum}, no digits of $x_n \alpha_n + y_n$ are determined, given $a_n$, if and only if
\[\begin{cases}
v_p(x_n a_n + y_n)\ge v_p(x_n) +1 & \text{if }n \equiv 0 \pmod{2}\\
v_p(x_n a_n + y_n)\ge v_p(x_n) & \text{if }n \equiv 1 \pmod{2}\\
\end{cases}.\]

The analysis of the M\"obius and bilinear fractional transformations in the case of Algorithm~\eqref{Alg: MR} is very similar to the one we presented for Ruban and \textit{Browkin I} in Section~\ref{Sec: Mobpadic} and Section~\ref{Sec: bilinearfractional}. In particular, we apply the results from Section~\ref{Sec: aux} and Proposition~\ref{Pro: digitsum} to this specific case of the functions $s$ and $t$. We explicitly state the results concerning the M\"obius transformation, as those for the bilinear fractional transformation can be derived in an analogous way. First, in the next two theorems, we present the necessary and sufficient conditions under which, given $s(\alpha)$ and $t(\alpha)$, one can determine $s\left (\frac{x\alpha+y}{z\alpha+t} \right )$ and $t\left (\frac{x\alpha+y}{z\alpha+t}\right )$, respectively.

\begin{Theorem}\label{thmfloor MRs}
Let $\alpha\in\Q_p$ and $x,y,z,t\in\Q$ with $xt - yz \ne 0$, and set $\gamma = \frac{x\alpha+y}{z\alpha+t}$. Let $v_p(\alpha)=-r\leq0$ and suppose that $s(\alpha)=\sum\limits_{n=-r}^{0}c_np^n$ is known. 
If $s(\alpha)=\alpha$, then we are able to compute $s(\gamma)$ and $t(\gamma)$. Suppose $s(\alpha) \ne \alpha$.
If either \[v_p(xs(\alpha)+y)\geq v_p(x)+1,\] 
or
\[v_p(zs(\alpha)+t)\geq v_p(z)+1,\]
then one cannot determine either $s(\gamma)$ or $t(\gamma)$. Otherwise, let us denote
\begin{equation*}
\begin{aligned}
v^{(1)}_s&=v_p(xs(\alpha)+y),\\
v^{(2)}_s&=v_p(zs(\alpha)+t),
\end{aligned}
\end{equation*}
and let
\[m=\min\{v_p(x)-v^{(1)}_s,v_p(z)-v^{(2)}_s\}.\]
Then, $s(\gamma)$ is uniquely determined by $s(\alpha)$ if and only if $m\geq v^{(2)}_s-v^{(1)}_s$ and $t(\gamma)$ is uniquely determined by $s(\alpha)$ if and only if $m\geq v^{(2)}_s-v^{(1)}_s-1$. In particular, in these cases,
\[s \left (\frac{x\alpha+y}{z\alpha+t}\right )=s\left (\frac{x s(\alpha)+y}{z s(\alpha)+t}\right ) \quad \text{and} \quad t \left (\frac{x\alpha+y}{z\alpha+t}\right )=t\left (\frac{x s(\alpha)+y}{z s(\alpha)+t}\right ).\]
\end{Theorem}

\begin{proof}
The proof proceeds as in Theorem~\ref{thmfloor}, using Proposition~\ref{Pro: digitsum}, Lemma~\ref{Lem: inverse}, and Lemma~\ref{Lem: prodab}.\end{proof}

\begin{Theorem}\label{thmfloor MRt}
Let $\alpha\in\Q_p$ and $x,y,z,t\in\Q$ with $xt - yz \ne 0$, and set $\gamma = \frac{x\alpha+y}{z\alpha+t}$. Let $v_p(\alpha)=-r\leq0$ and suppose that $t(\alpha)=\sum\limits_{n=-r}^{-1}c_np^n$ is known. 
If $t(\alpha)=\alpha$, then we are able to compute $s(\gamma)$ and $t(\gamma)$. Suppose $t(\alpha) \ne \alpha$.
If either \[v_p(xt(\alpha)+y)\geq v_p(x),\] 
or
\[v_p(zt(\alpha)+t)\geq v_p(z),\]
then one cannot determine either $s(\gamma)$ or $t(\gamma)$. Otherwise, let us denote
\begin{equation*}
\begin{aligned}
v^{(1)}_t&=v_p(xt(\alpha)+y),\\
v^{(2)}_t&=v_p(zt(\alpha)+t),
\end{aligned}
\end{equation*}
and let
\[m=\min\{v_p(x)-v^{(1)}_t,v_p(z)-v^{(2)}_t\}.\]
Then, $s(\gamma)$ is uniquely determined by $t(\alpha)$ if and only if $m\geq v^{(2)}_t-v^{(1)}_t+1$ and $t(\gamma)$ is uniquely determined by $t(\alpha)$ if and only if $m\geq v^{(2)}_t-v^{(1)}_t$. In particular, in these cases,
\[s \left (\frac{x\alpha+y}{z\alpha+t}\right )=s\left (\frac{x t(\alpha)+y}{z t(\alpha)+t}\right ) \quad \text{and} \quad t \left (\frac{x\alpha+y}{z\alpha+t}\right )=t\left (\frac{x t(\alpha)+y}{z t(\alpha)+t}\right ).\]
\end{Theorem}
\begin{proof}
The proof proceeds as in Theorem~\ref{thmfloor}, using Proposition~\ref{Pro: digitsum}, Lemma~\ref{Lem: inverse}, and Lemma~\ref{Lem: prodab}.\end{proof}

Following Section~\ref{Sec: MobInp}, we examine how using more partial quotients of $\alpha$ contributes to computing the functions $s$ and $t$ of a M\"obius transformation, and we adopt the same notation for the coefficients resulting from the input transformations (see Definition~\ref{Def: recur}).
Lemma~\ref{Lem: rec} also holds for the continued fraction expansion obtained by Algorithm~\eqref{Alg: MR}.

\begin{Lemma}\label{Lem: MRbeforegood}
Let $\{a_n\}_{n\ge0}$ be a sequence of $p$-adic numbers such that \eqref{Eq: mrvaluation} holds for all $n\ge0$, and let $x_0, y_0 \in \Q_p$. Define
\[
\mu_n = \min\{ v_p(x_n a_n), v_p(y_n) \},
\]
where the sequences $\{x_n\}_{n\ge0}$ and $\{y_n\}_{n\ge0}$ are defined as in Definition~\ref{Def: recur}.  
If
\[
\mu_n \le v_p(x_n a_n + y_n) \le \mu_n + r_n, \quad \text{for some } n \equiv 0 \pmod{2},
\]
or
\[
\mu_n \le v_p(x_n a_n + y_n) \le \mu_n + r_n - 1, \quad \text{for some } n \equiv 1 \pmod{2},
\]
then 
\[
v_p(x_m a_m) < v_p(y_m) \quad \text{for all } m \ge n+1.
\]
\end{Lemma}
\begin{proof}
If $n \equiv 0 \pmod{2}$ then $a_n = s(\alpha_n)$ and we know $r_{n}+1$ digits of $\alpha_n$. If $v_p(x_na_n+y_n)=\mu_n$, then
\[v_p(a_{n+1}x_{n+1}) = v_p(a_{n+1}) + \mu_n \le v_p(a_{n+1}) + v_p(a_{n}) + v_p(x_{n})<v_p(x_{n})=v_p(y_{n+1}),\]
since $v_p(a_n)+v_p(a_{n+1})<0$ for all $n\ge0$, by \eqref{Eq: mrvaluation}.
If $\mu_n<v_p(x_na_n+y_n)\le\mu_n+r_n$, then
\[v_p(a_{n+1}x_{n+1}) \le \mu_n + r_n - r_{n+1} = v_p(x_n) - r_n + r_n - r_{n+1}<v_p(x_{n})=v_p(y_{n+1}),\]
since $v_p(a_{n+1})<0$, by \eqref{Eq: mrvaluation}.

If $n \equiv 1 \pmod{2}$ then $a_n = t(\alpha_n)$ and we know $r_{n}$ digits of $\alpha_n$. Using the same reasoning, if $v_p(x_na_n+y_n)=\mu_n$, then $v_p(a_{n+1}x_{n+1})<v_p(y_{n+1})$. If $\mu_n<v_p(x_na_n+y_n)\le\mu_n+r_n-1$, then
\[v_p(a_{n+1}x_{n+1}) \le \mu_n +r_n -1 - r_{n+1} = v_p(x_n) - r_n +r_n -1 - r_{n+1}<v_p(x_n)=v_p(y_{n+1}),\]
and the claim follows.
\end{proof}
Using the same reasoning of Proposition~\ref{Prop: othername}, we obtain that, if 
\[\begin{aligned}v_p(x \alpha + y) &\le v_p(y)\\
v_p(z \alpha + t) &\le v_p(z)
\end{aligned}\]
then also for Algorithm~\eqref{Alg: MR} we have
\[v_n^{(2)}-v_n^{(1)} = v_p(z_1)-v_p(x_1) \quad \text{for all }n \ge 1.\]

Hence, once we are in a \emph{good} case, we know that $v_n^{(2)}-v_n^{(1)}$ is constant for all $n\ge1$. Therefore, the problem of computing the $s$ and $t$ functions
\[s(\gamma)=s \left (\frac{x\alpha+y}{z\alpha+t}\right ) \quad \text{and} \quad t(\gamma)=t \left (\frac{x\alpha+y}{z\alpha+t}\right )\]
depends on the existence of an index $n\ge0$ for which $\alpha_n$ has sufficiently small valuation. Indeed, the conditions of Theorem~\ref{thmfloor MRs} and Theorem~\ref{thmfloor MRt} for the computation of $s(\gamma)$ and $t(\gamma)$ become: $s(\gamma)$ is determined if and only if exists $n\ge 1$
\[\begin{cases}- v_p(\alpha_n) \ge v_p(z_1)-v_p(x_1) &\text{if }n \equiv 0 \pmod{2} \\
- v_p(\alpha_n) \ge v_p(z_1)-v_p(x_1) + 1 &\text{if }n \equiv 1 \pmod{2}\end{cases}\]
and $t(\gamma)$ determined if and only if, for $n\ge 1$
\[\begin{cases}- v_p(\alpha_n) \ge v_p(z_1)-v_p(x_1) -1 &\text{if }n \equiv 0 \pmod{2} \\
- v_p(\alpha_n) \ge v_p(z_1)-v_p(x_1) &\text{if }n \equiv 1 \pmod{2}\end{cases}.\]

The final problematic case is when no digit of the numerator or the denominator of $\gamma$ can be determined after any number of input transformations. In this case, we should have
\begin{equation}\label{Eq: condvxMR}
\begin{aligned}&v_p(x_{n+1})\geq v_p(x_{n})+1 \quad &\text{for all }n \equiv 0 \pmod{2},\\
&v_p(x_{n+1})\geq v_p(x_{n}) \quad &\text{for all }n \equiv 1 \pmod{2}.
\end{aligned}
\end{equation}
Using a similar argument to Proposition~\ref{Prop: tolet}, it is possible to prove that in this case we have $x\alpha+y=0$, hence $\alpha=-\frac{y}{x}$. This is clearly not possible, since the continued fraction of rational numbers via Algorithm~\eqref{Alg: MR} is finite, by Proposition~\ref{Prop: finexpansion}. Therefore, after a finite number of input transformations it is possible to compute $\lfloor x \alpha + y \rfloor_p$ when $\alpha = -\frac{y}{x}$. The structure of the algorithm for the computation of the partial quotients of the M\"obius transformation is similar to the one for Ruban and \textit{Browkin I}. By Corollary~\ref{Coro: MRvalunbounded}, for $\mu$-almost all $\alpha \in \Q_p$, the partial quotients at odd positions have unbounded $p$-adic valuation, where $\mu$ is the Haar measure. Therefore, also for Algorithm~\eqref{Alg: MR}, given $x,y,z,t \in \Q$, for almost all $\alpha \in \Q_p$, after a finite number of input transformations we compute the floor function (either via $s$ or $t$).

\section{Computational analysis}\label{Sec: Computations}
In this section, we analyze the performances of Algorithm~\ref{Alg: Gosp1} and Algorithm~\ref{Alg: Gosp2} for computing, respectively, the $p$-adic continued fraction of the M\"obius transformation
\begin{equation*}
\gamma =\frac{x\alpha+y}{z\alpha+t},
\end{equation*}
and the bilinear fractional transformation 
\begin{equation*}
\gamma=\frac{x\alpha\beta+y\alpha+z\beta+t}{e\alpha\beta+f\alpha+g\beta+h},
\end{equation*}
given the $p$-adic continued fractions of $\alpha=[a_0,a_1,\ldots]$ and $\beta=[b_0,b_1,\ldots]$ and $x,y,z,t,e,f,g,h\in\Q$. The purpose of this section is to study the behavior of these algorithms, mainly to understand how many partial quotients from the input continued fractions are usually required in order to get an output partial quotient for the $p$-adic continued fraction of the transformation. For the simulations, we use $p$-adic quadratic irrational numbers that do not appear to have a periodic $p$-adic continued fractions. 
In fact, for algebraic irrationals, the distribution of $p$-adic digits and partial quotients is largely believed to be the same shared by \textit{almost all} $p$-adic numbers, as presented in \cite{adamczewski2007complexity,sun2010borel}. Algorithm~\ref{Alg: Gosp1} and Algorithm~\ref{Alg: Gosp2}, together with the computational experiments contained in this section, have been implemented in SageMath and the code is publicly available\footnote{\href{https://github.com/giulianoromeont/p-adic-continued-fractions}{https://github.com/giulianoromeont/p-adic-continued-fractions}}.

\subsection{M\"obius transformation}\label{Sec: compMob}
The procedure for computing the $p$-adic continued fraction of the M\"obius transformation is described in Algorithm~\ref{Alg: Gosp1}. In essence, we perform input transformations iteratively until we obtain
\begin{equation*}
\frac{x_{\overline{n}}\alpha_{\overline{n}}+y_{\overline{n}}}{z_{\overline{n}}\alpha_{\overline{n}}+t_{\overline{n}}},
\end{equation*}
such that $v_p(x_{\overline{n}}\alpha_{\overline{n}})<v_p(y_{\overline{n}})$ and $v_p(z_{\overline{n}}\alpha_{\overline{n}})<v_p(t_{\overline{n}})$. By Corollary~\ref{Cor: valunique} and Proposition~\ref{Prop: tolet}, we know that this is always the case after finitely many input transformations, as long as both $x\alpha+y$ and $z\alpha+t$ are different from zero. At this point, by Proposition~\ref{Prop: othername}, we can determine the partial quotient of the M\"obius transformation if and only if 
\begin{equation}\label{Eq: satiscond}
v_p(a_n)\leq v_p(\overline{x})-v_p(\overline{z}),
\end{equation}
for some $n\geq\overline{n}$. As already shown in Example~\ref{Exa: neveroutput}, this condition may never be satisfied. However, thanks to the metric results of Section~\ref{Sec: measure}, we know that the set of $p$-adic numbers with such property has Haar measure $0$. 

In Figure~\ref{Fig: rubmob95}, we analyze the number of input partial quotients that are required to compute up to $1000$ output partial quotients. We consider $100$ different M\"obius transformation of the $p$-adic Ruban's continued fraction of $\sqrt{95}$ in $\mathbb{Q}_{13}$. The coefficients of the Möbius transformation are chosen from the set $\{0, \ldots, 10000\}$ so that the transformation is non-singular. In Figure~\ref{Fig: rubmob95}, it is possible to notice several horizontal lines forming a staircase. These lines correspond to the input partial quotients $a_n$, with $n=51,216,455,873,1036,2540$. The reason is that these partial quotients have unusually small valuation:
\begin{align*}
v_p(a_{51})&=v_p\left(\frac{2599}{13^3}\right)=-3,& v_p(a_{216})&=v_p\left(\frac{27212}{13^3}\right)=-3,\\ 
v_p(a_{455})&=v_p\left(\frac{4818}{13^3}\right)=-3,&
v_p(a_{873})&=v_p\left(\frac{4623}{13^3}\right)=-3,\\
v_p(a_{1036})&=v_p\left(\frac{16382}{13^3}\right)=-3,&
v_p(a_{2540})&=v_p\left(\frac{177}{13^4}\right)=-4.
\end{align*}

In fact, for these input partial quotients, the output condition~\eqref{Eq: satiscond} is more likely to be satisfied. Moreover it seems that in many cases, after the output transformation, we still end up in the good case with $v_p(x_{n}\alpha_{n})<v_p(y_{n})$ and $v_p(z_{n}\alpha_{n})<v_p(t_{n})$. In this case, the output condition is again of the form \eqref{Eq: satiscond}, so that the same partial quotient $a_n$ is used to check whether the condition holds. This is the reason why, once we encounter a partial quotient with ``very negative" valuation, it is likely to perform several consecutive outputs.

\begin{center}
\begin{figure}[H]
\includegraphics[scale=0.75]{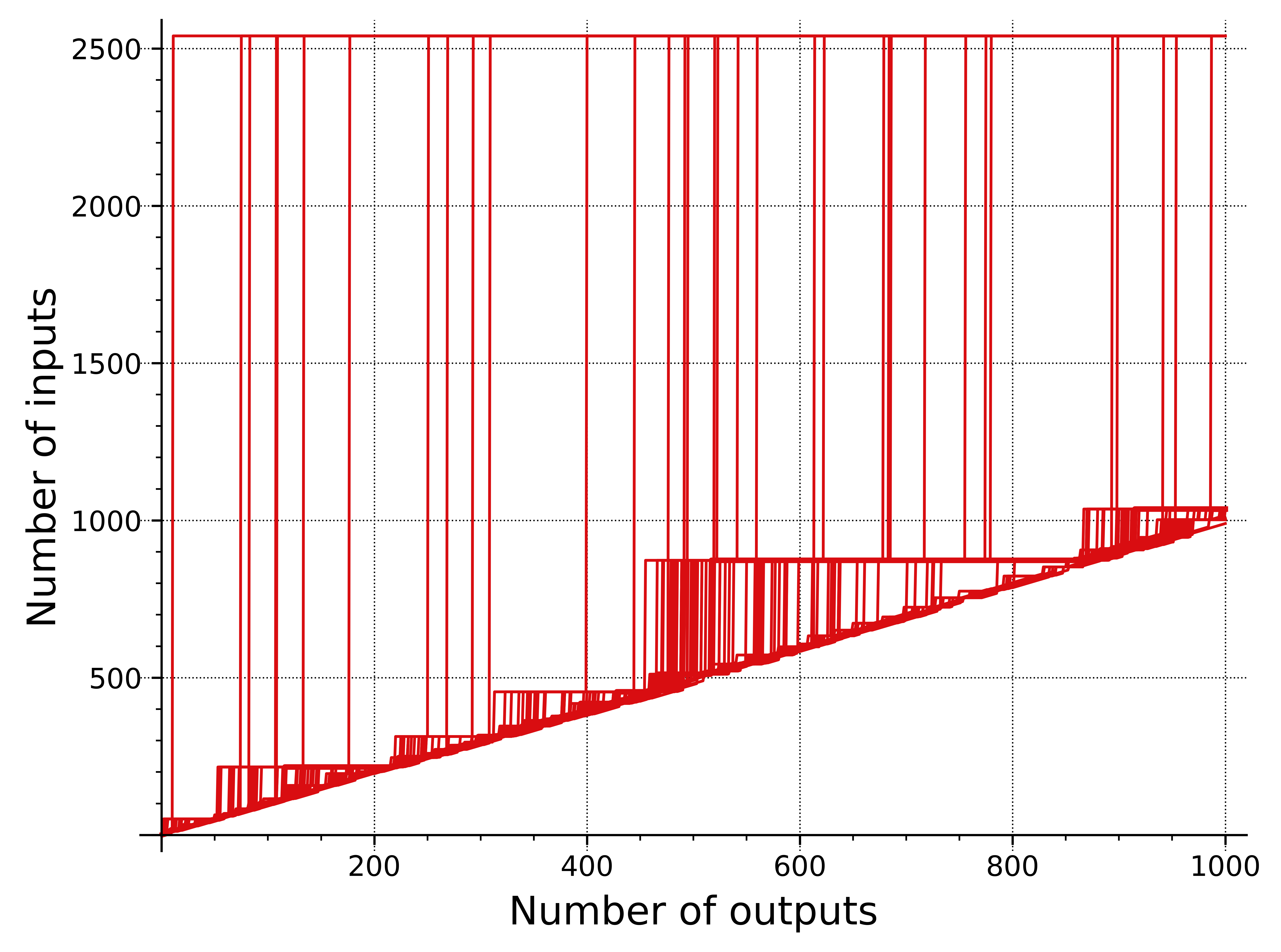}
\caption{Number of inputs required to compute up to $1000$ partial quotients for the $p$-adic Ruban's continued fraction of $100$ different M\"obius transformations of $\sqrt{95}$ in $\mathbb{Q}_{13}$.}
\label{Fig: rubmob95}
\end{figure}
\end{center}

A similar result can be observed in Figure~\ref{Fig: rubmob16653Q2137} for the Ruban's continued fraction of $100$ different M\"obius transformations with coefficients in $\{0,\ldots,10000\}$ for $\sqrt{16653}$ in $\mathbb{Q}_{2137}$. In this case, it is possible to see a clear line corresponding to $a_{1903}$, that is the first partial quotient with valuation $-2$. In fact, for large $p$, having valuation $v_p(a_n)\leq -2$ is rare in general, as it happens around $\frac{1}{p}$ times.

\begin{center}
\begin{figure}[H]
\includegraphics[scale=0.75]{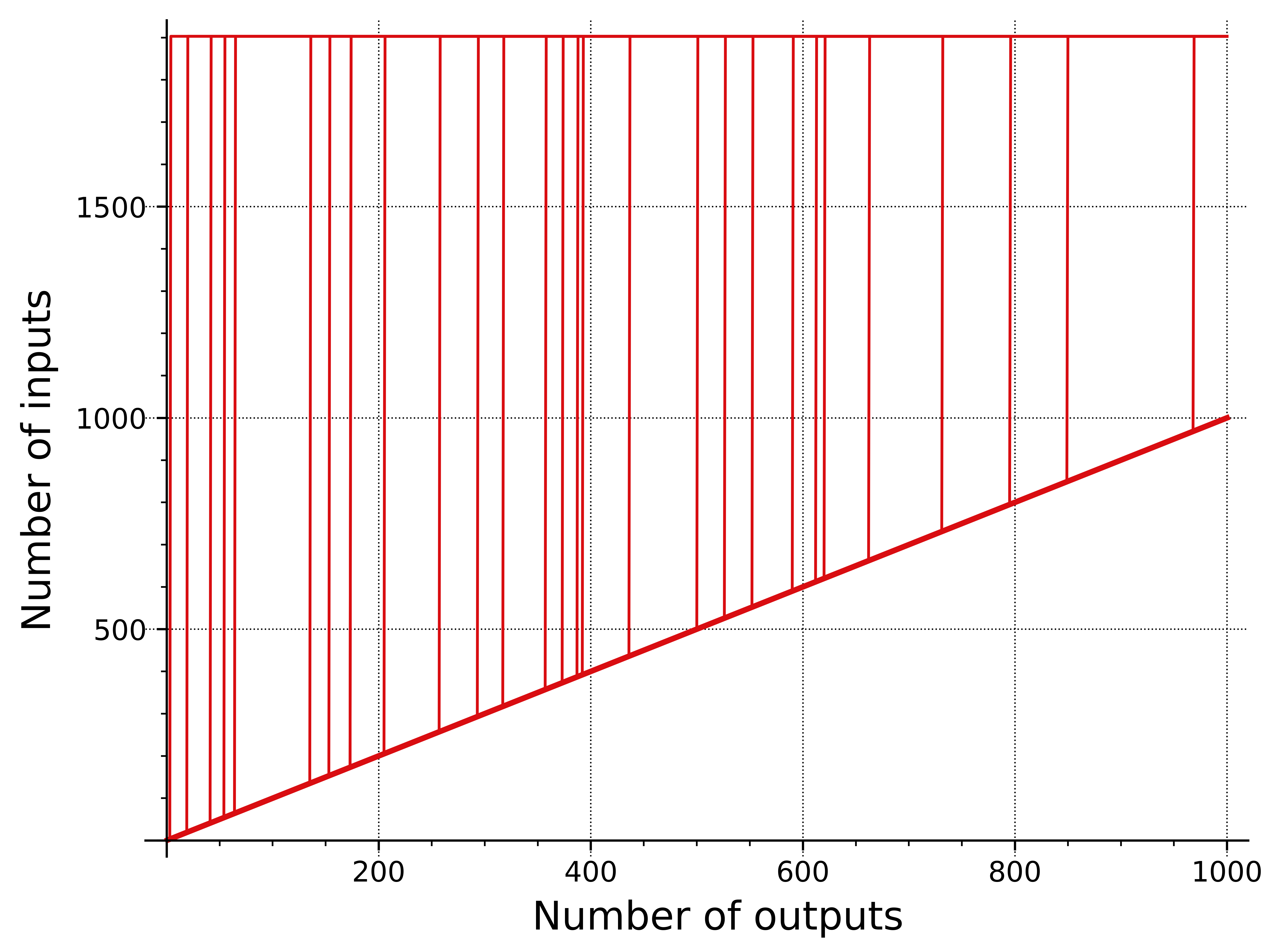}
\caption{Number of inputs required to compute up to $1000$ partial quotients for the $p$-adic Ruban's continued fraction of $100$ different M\"obius transformations of $\sqrt{16653}$ in $\mathbb{Q}_{2137}$.}
\label{Fig: rubmob16653Q2137}
\end{figure}
\end{center}

In Figure~\ref{Figure: BrowkinMobius}, we report the same analysis for Browkin's $p$-adic continued fraction. We plot $100$ different M\"obius transformations of \textit{Browkin I} continued fraction of $\sqrt{95}$ in $\mathbb{Q}_{13}$ and $\sqrt{16653}$ in $\mathbb{Q}_{2137}$.

\begin{figure}[H]
\centering
\begin{subfigure}{0.49\textwidth}
  \includegraphics[width=\linewidth]{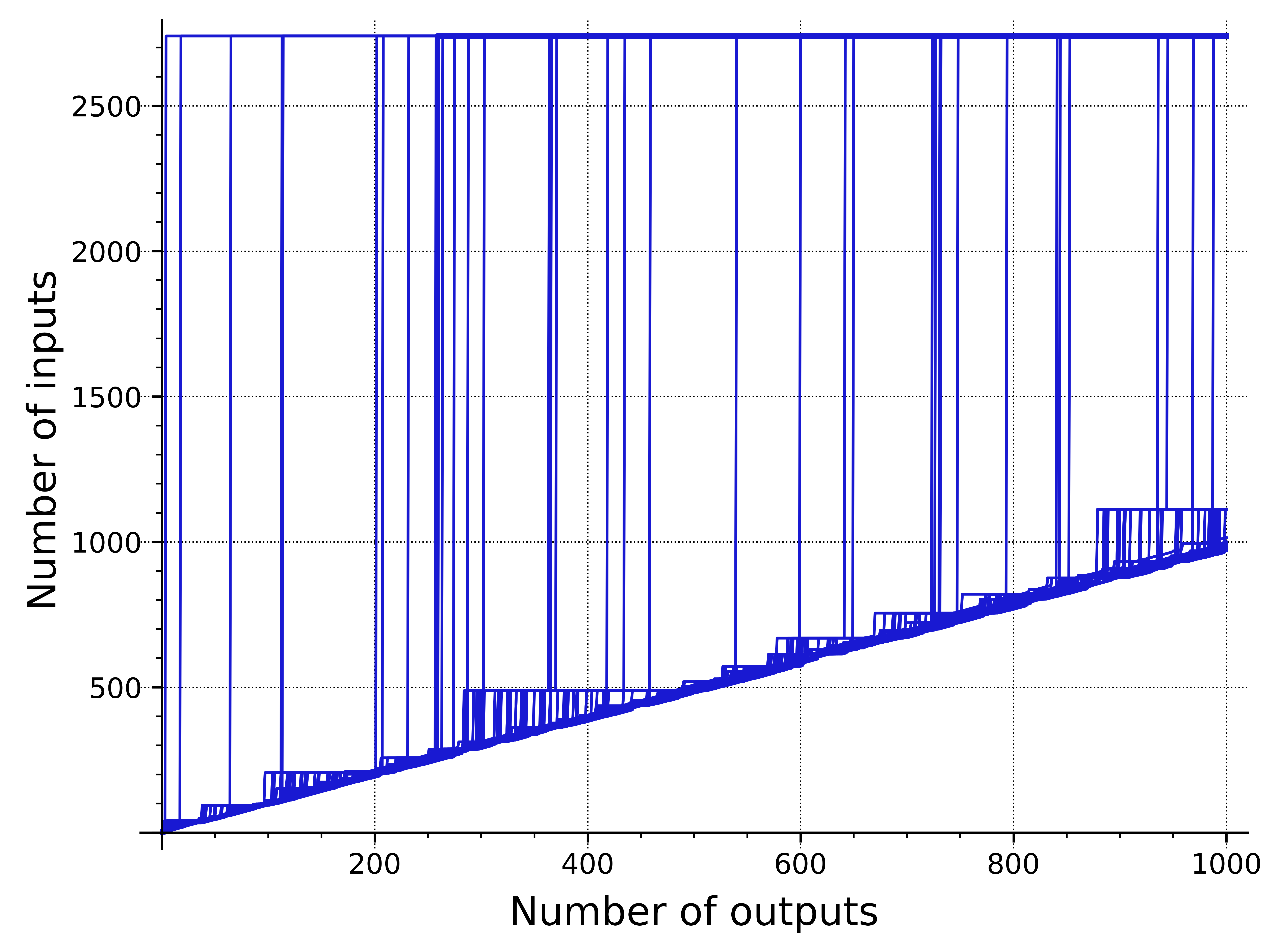}
\end{subfigure}
\hspace{0.001\textwidth}
\begin{subfigure}{0.49\textwidth}
  \includegraphics[width=\linewidth]{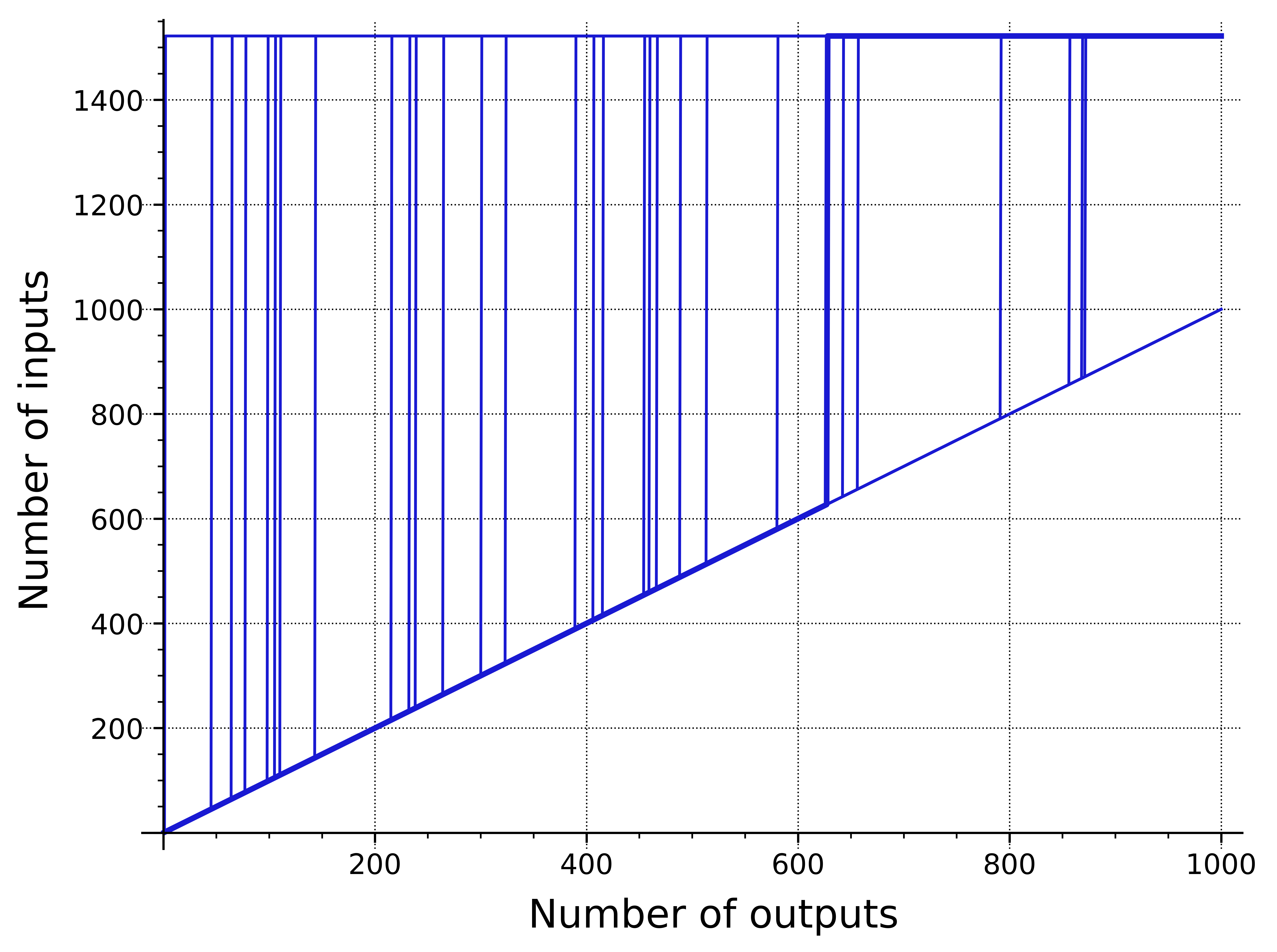}
\end{subfigure}

\caption{Number of inputs required to compute up to $1000$ partial quotients for the $p$-adic Browkin's continued fraction of $100$ different M\"obius transformations of $\sqrt{95}$ in $\mathbb{Q}_{13}$, on the left, and of $\sqrt{16653}$ in $\mathbb{Q}_{2137}$, on the right.}
\label{Figure: BrowkinMobius}
\end{figure}
In Figure~\ref{Figure: BrowkinMobius} there are very clear lines in both plots, corresponding to partial quotients of low $p$-adic valuation. For the continued fraction of $\sqrt{95}$ in $\mathbb{Q}_{13}$, on the left plot, the top line corresponds to
\[v_p(a_{2740})=v_p\left(-\frac{125212}{13^4}\right)=-4.\]
For Browkin's $2137$-adic continued fraction, of $\sqrt{16653}$ the top line corresponds to the first partial quotient of valuation $-2$, that is $a_{1522}$. Moreover, notice that in all these cases, both for Browkin's and Ruban's continued fractions, the graphic tends to follow and to lie above the line of slope $1$.

\subsection{Bilinear fractional transformation}\label{Sec: compbil}
At the end of Section~\ref{Sec: inptrans} and in Section~\ref{Sec: thebilalgorithm}, we have outlined the algorithm to compute the $p$-adic partial quotients for the bilinear fractional transformations of two $p$-adic numbers $\alpha$ ad $\beta$. The first step of the algorithm consists in performing $\beta$-input transformations until we get
\begin{equation*}
\frac{\alpha(x_n\beta_n+y_n)+(z_n\beta_n+t_n)}{\alpha(e_n\beta_n+f_n)+(g_n\beta_n+h_n)},
\end{equation*}
with 
\[v_p(x_n\beta_n)<v_p(y_n), \ \  v_p(z_n\beta_n)<v_p(t_n),  \ \ v_p(e_n\beta_n)<v_p(f_n), \ \ v_p(g_n\beta_n)<v_p(h_n).\]
This step of the algorithm has basically the same complexity as the first step of Algorithm~\ref{Alg: Gosp1}. After that, if
\begin{align*}
    v_p(x_n\alpha+z_n)&=\min\{v_p(x_n\alpha),v_p(z_n)\},\\
    v_p(e_n\alpha+g_n)&=\min\{v_p(e_n\alpha),v_p(g_n)\},
\end{align*}
then we are in the optimal situation, otherwise we swap the role of $\alpha$ and $\beta$ and we run again the algorithm. Whenever we have
\begin{align*}
&v_p(x_{\overline{n}}\beta)<v_p(y_{\overline{n}}),  & 
&v_p(z_{\overline{n}}\beta)<v_p(t_{\overline{n}}),   &
&v_p(x_{\overline{n}}\alpha)<v_p(z_{\overline{n}}),\\
&v_p(e_{\overline{n}}\beta)<v_p(f_{\overline{n}}),   &
&v_p(g_{\overline{n}}\beta)<v_p(h_{\overline{n}}),   &
&v_p(e_{\overline{n}}\alpha)<v_p(g_{\overline{n}}),
\end{align*}
then we are able to compute the output partial quotient if and only if
\[\max \{v_p(a_{n_\alpha}),v_p(b_{n_\beta})\}\leq v_p(x_{\overline{n}})-v_p(e_{\overline{n}}),\]
for some $n_\alpha,n_\beta\geq\overline{n}$.
The latter condition may be never satisfied, but this undesired situation happens only whenever the partial quotients of either $\alpha$ or $\beta$ have bounded $p$-adic valuation. Again, from the results of Section~\ref{Sec: measure}, we know that \textit{almost all} $p$-adic numbers, with respect to Haar measure, have unbounded valuation. Figures~\ref{Fig: RubBil} and~\ref{Fig: BroBil} show the number of $\alpha$-input and $\beta$-input transformations required to obtain up to $10000$ partial quotients of a bilinear fractional transformation of $7$-adic continued fraction of $(\alpha,\beta)=(\sqrt{79},\sqrt{151})$ in $\mathbb{Q}_{7}$, using Ruban’s and Browkin’s algorithms, respectively.
As for the M\"obius transformation, the number of inputs tends to grow linearly with the number of outputs. Horizontal lines in the plots correspond to partial quotients of $\alpha$ or $\beta$ with low $p$-adic valuation.

\begin{center}
\begin{figure}[H]
\includegraphics[scale=0.71]{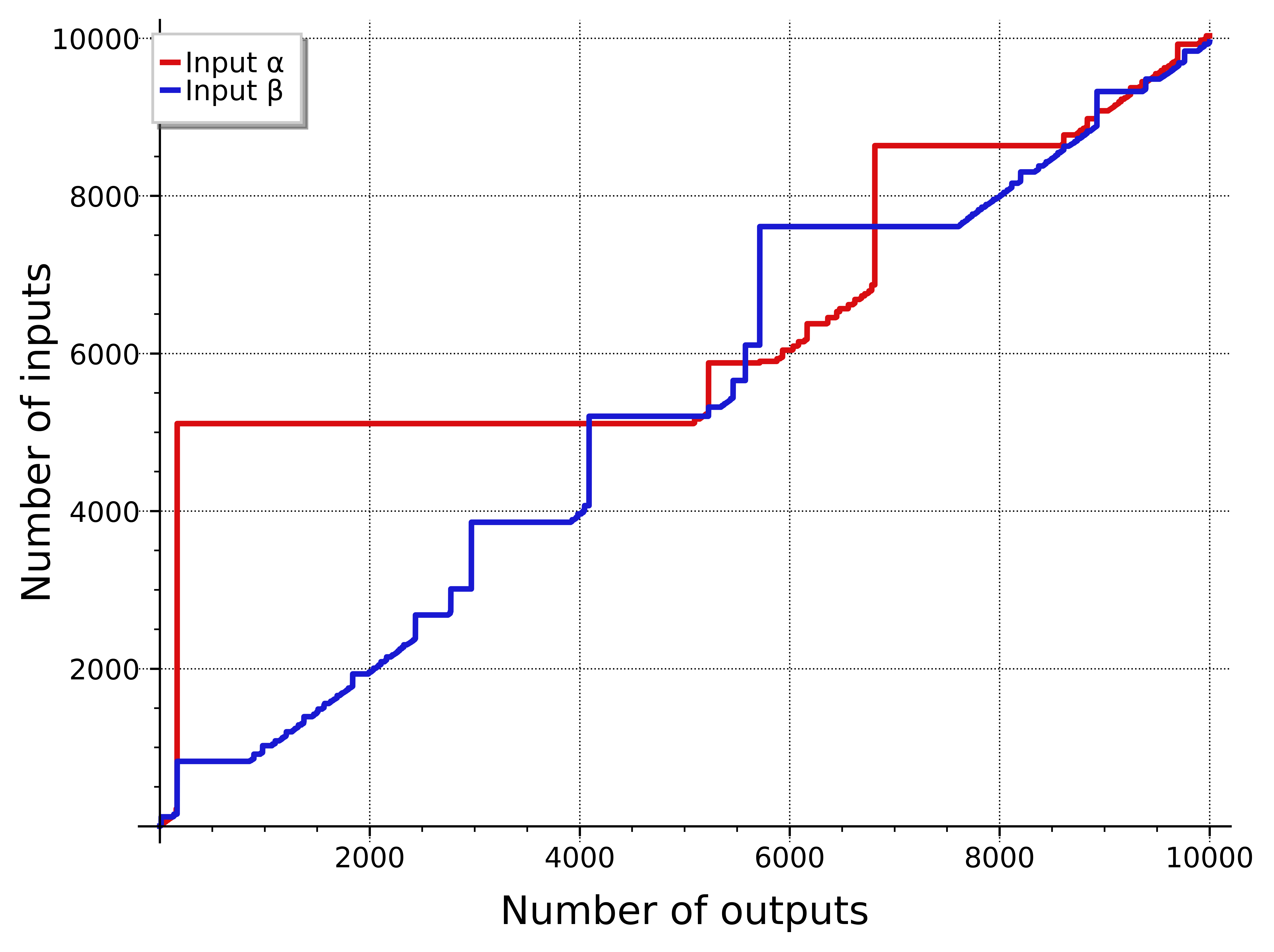}
\caption{Number of $\alpha$-inputs and $\beta$-inputs required to compute up to $10000$ partial quotients for Ruban's continued fraction of a bilinear fractional transformations of $(\alpha,\beta)=(\sqrt{79},\sqrt{151})$ in $\mathbb{Q}_{7}$.}
\label{Fig: RubBil}
\end{figure}
\end{center}

\begin{center}
\begin{figure}[H]
\includegraphics[scale=0.71]{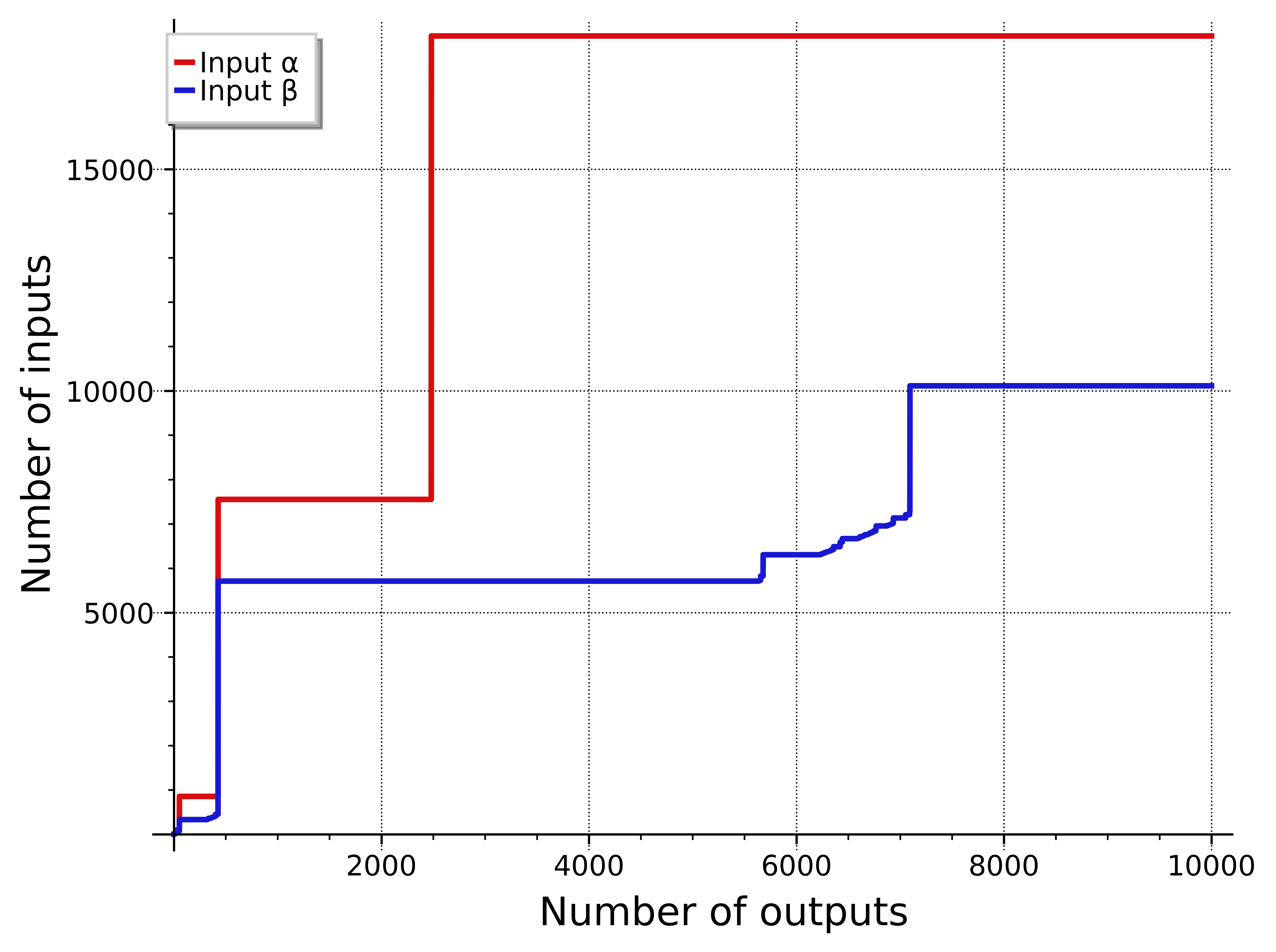}
\caption{Number of $\alpha$-inputs and $\beta$-inputs required to compute up to $10000$ partial quotients for \textit{Browkin I} continued fraction of a bilinear fractional transformation of $(\alpha,\beta)=(\sqrt{79},\sqrt{151})$ in $\mathbb{Q}_{7}$.}
\label{Fig: BroBil}
\end{figure}
\end{center}

\section*{Acknowledgments}

The authors are members of GNSAGA of INdAM.\\
G.R. was supported by the project PRIMUS/25/SCI/008 from Charles University.
G.R. acknowledges that this study was carried out within the MICS (Made in Italy – Circular and Sustainable) Extended Partnership and received funding from the European Union Next-Generation EU (Piano Nazionale di Ripresa e Resilienza (PNRR) – Missione 4 Componente 2, Investimento 1.3 – D.D. 1551.11-10-2022, PE00000004).

\printbibliography

\section*{Appendix}

\IncMargin{1.5em}
\begin{algorithm}[H]\label{Alg: Gosp1}
	\caption{Computation of the $p$-adic continued fraction of the Möbius transformation $\tfrac{x\alpha+y}{z\alpha+t}$, given the rational coefficients, the partial quotients of $\alpha$, and an iteration bound $N$.}\label{gosper1}
	\SetKwData{Left}{left}
	\SetKwData{This}{this}
	\SetKw{And}{and}
	\SetKwFunction{Union}{Union}
	\SetKwFunction{FindCompress}{FindCompress}
	\SetKwInOut{Input}{Input}
	\SetKwInOut{Output}{Output}
    \SetKw{KwTo}{to}
    \SetKwComment{Comment}{$\triangleright$ }{ }
	\Input{$[a_0, a_1, \ldots]=\alpha$, $x,y,z,t \in \mathbb Q$, $N \in \mathbb N$}
	\Output{$[l_0, l_1, \ldots] = \frac{x\alpha+y}{z\alpha+t}$}
	\BlankLine
	$i \gets 0$, $j \gets 0$, $k \gets 0$\\
    
    \ 
\While{$i<N$}{
     \While{$i<N$ \And $(v_p(xa_j+y) > v_p(x) \mathrm{ \mathbf{\  or \ } } v_p(za_j+t) > v_p(z))$}{ 
     $\begin{pmatrix} x & y \cr z & t \end{pmatrix} \gets \begin{pmatrix} x & y \cr z & t \end{pmatrix}\begin{pmatrix} a_j & 1 \cr 1 & 0 \end{pmatrix}$\\
            
    $j \gets j + 1, \ i \gets i + 1$}
    \If{$v_p(xa_j)\geq v_p(y) \mathrm{ \mathbf{\  or \ } } v_p(za_j)\geq v_p(t)$}{
    $\begin{pmatrix} x & y \cr z & t \end{pmatrix} \gets \begin{pmatrix} x & y \cr z & t \end{pmatrix}\begin{pmatrix} a_j & 1 \cr 1 & 0 \end{pmatrix}$

    $j \gets j + 1,\ i \gets i + 1$}
            
    $u\gets v_p(z)-v_p(x)$\\   
   \If{$u<0$}{$l_k\gets 0$
   
   $k\gets k+1$
   }
   \Else{

        \While{$i<N$ \And $-v_p(a_j)< u$}{
        $\begin{pmatrix} x & y \cr z & t \end{pmatrix} \gets \begin{pmatrix} x & y \cr z & t \end{pmatrix}\begin{pmatrix} a_j & 1 \cr 1 & 0 \end{pmatrix}$

            $j \gets j + 1$, $i \gets i + 1$
            }

            $l_k \gets \left\lfloor \frac{xa_j+y}{z a_j+t}\right\rfloor_p$
        
            $\begin{pmatrix} x & y \cr z & t \end{pmatrix} \gets \begin{pmatrix} 0 & 1 \cr 1 & -l_k \end{pmatrix}\begin{pmatrix} x & y \cr z & t \end{pmatrix}$

            $k \gets k + 1$,  $i \gets i+1$}}

\end{algorithm}

\resizebox{0.86\textwidth}{!}{
\begin{algorithm}[H]
\caption{Computation of the $p$-adic continued fraction of the bilinear fractional transformation $\frac{x\alpha\beta+y\alpha+z\beta+t}{e\alpha\beta+f\alpha+g\beta+h}$, given its rational coefficients, the partial quotients of $\alpha$ and $\beta$, and an iteration bound $N$.}\label{Alg: Gosp2}
	\SetKwData{Left}{left}
	\SetKwData{This}{this}
	\SetKw{And}{and}
	\SetKwFunction{Union}{Union}
	\SetKwFunction{FindCompress}{FindCompress}
	\SetKwInOut{Input}{Input}
	\SetKwInOut{Output}{Output}
    \SetKw{KwTo}{to}
    \SetKwComment{Comment}{$\triangleright$ }{ }
	\Input{$[a_0, a_1, \ldots]=\alpha$, $[b_0, b_1, \ldots]=\beta$, $x,y,z,t,e,f,g,h \in \mathbb Q$, $N \in \mathbb N$}
	\Output{$[l_0, l_1, \ldots] = \frac{x\alpha\beta+y\alpha+z\beta+t}{e\alpha\beta+f\alpha+g\beta+h}$}
	\BlankLine
	$i \gets 0$, $j_\alpha \gets 0$, $j_\beta \gets 0$, $k \gets 0$\\

\While{$i<N$}{
     \While{$i<N$ \And $(v_p(xb_{j_\beta}+y) > v_p(x) \mathrm{ \mathbf{\  or \ } } v_p(zb_{j_\beta}+t) > v_p(z) \mathrm{ \mathbf{\  or \ } } v_p(eb_{j_\beta}+f) > v_p(e) \mathrm{ \mathbf{\  or \ } } v_p(gb_{j_\beta}+h) > v_p(g))$}{ \ \\
     $\begin{pmatrix}
    x & y & z & t \\
    e & f & g & h
\end{pmatrix} \gets \begin{pmatrix}
    xb_{j_\beta}+y & x & zb_{j_\beta}+t & z \\
    eb_{j_\beta}+f & e & gb_{j_\beta}+h & g 
\end{pmatrix}$\\
            
    $j_\beta \gets j_\beta + 1, \ i \gets i + 1$}
    $\begin{pmatrix}
    x & y & z & t \\
    e & f & g & h
\end{pmatrix} \gets \begin{pmatrix}
    xb_{j_\beta}+y & x & zb_{j_\beta}+t & z \\
    eb_{j_\beta}+f & e & gb_{j_\beta}+h & g 
\end{pmatrix}$\\
            
    $j_\beta \gets j_\beta + 1, \ i \gets i + 1$\\
            
\BlankLine
  
   \If{$v_p(xa_{j_a}+z)=\min\{v_p(xa_{j_a}),v_p(z)\}$ \And $v_p(ea_{j_a}+g)=\min\{v_p(ea_{j_a}),v_p(g)\}$}{
   \BlankLine
   $\begin{pmatrix}
    x & y & z & t \\
    e & f & g & h
\end{pmatrix} \gets \begin{pmatrix}
    xa_{j_\alpha}+z & ya_{j_\alpha}+t & x & y \\
    ea_{j_\alpha}+g & fa_{j_\alpha}+h & e & f 
\end{pmatrix}$

 $u=v_p(e)-v_p(x)$

   \If{$u<0$}{$l_k\gets 0$
   
   $k\gets k+1$
   }
   \Else{

        \While{ $i<N$ \And $-v_p(a_{j_\alpha})< u$}{
            
            $\begin{pmatrix}
    x & y & z & t \\
    e & f & g & h
\end{pmatrix} \gets \begin{pmatrix}
    xa_{j_\alpha}+z & ya_{j_\alpha}+t & x & y \\
    ea_{j_\alpha}+g & fa_{j_\alpha}+h & e & f 
\end{pmatrix}$
            
    $j_\alpha \gets j_\alpha + 1, \ i \gets i + 1$        
            }
       \While{ $i<N$ \And $-v_p(b_{j_\beta})< u$}{
            
            $\begin{pmatrix}
    x & y & z & t \\
    e & f & g & h
\end{pmatrix} \gets \begin{pmatrix}
    xb_{j_\beta}+y & x & zb_{j_\beta}+t & z \\
    eb_{j_\beta}+f & e & gb_{j_\beta}+h & g 
\end{pmatrix}$
$j_\beta \gets j_\beta + 1, \ i \gets i + 1$
}

$l_k \gets \left\lfloor \frac{xa_{j_\alpha }b_{j_\beta}+ya_{j_\alpha }+z b_{j_\beta}+t}{ea_{j_\alpha }b_{j_\beta}+fa_{j_\alpha }+g b_{j_\beta}+h}\right\rfloor_p$

        $\begin{pmatrix}
    x & y & z & t \\
    e & f & g & h
\end{pmatrix}\gets\begin{pmatrix}
    e & f & g & h \\
    x-l_ke & y-l_kf & z-l_kg & t-l_kh 
\end{pmatrix}  $

    $i \gets i+1$, $k \gets k+1$}}
    \Else{

        \text{swap role} $\alpha\leftrightarrow\beta$,  
    
    }   }
\end{algorithm}}

\end{document}